\documentclass[12pt]{amsart}
\usepackage{epic,latexsym,amsbsy,amsmath,amscd,amssymb}
\newcommand\AND{\quad\text{and}\quad}

\newcommand\bd{\partial}

\newcommand\Bss{\ {\buildrel \approx \over B}}
\newcommand\C{\mathbb C}
\newcommand\Ccal{\mathcal C}
\newcommand\Dcal{\mathcal D}

\newcommand\diam{\operatorname{\rm diam}}

\newcommand\dsf{\mathsf{d}} 
\newcommand\Ecal{\mathcal E}
\newcommand\ee{\mathbf{e}}
\newcommand\EE{E}  
\newcommand\ep{\epsilon}
\newcommand\F{\mathbb F}
\newcommand\ff{\mathbf{f}}
 
\newcommand\impliess{\buildrel * \over \implies}
\newcommand\K{\mathbb K}
\newcommand\la{\lambda}
\newcommand\N{\mathbb N}
\newcommand\ol{\overline}
\newcommand\Pb{\mathbf{P}}

\newcommand\Pol{\mathfrak P}

\newcommand\pre{\preccurlyeq}
\newcommand\Prob{\mathsf{Pr}}

\newcommand\qss{\ {\buildrel \approx \over q}}
\newcommand\Qss{\ {\buildrel \approx \over Q}}
\newcommand\R{\mathbb R}

\newcommand\si{\sigma}
\newcommand\Si{\mathbf{\Sigma}}

\newcommand\Tcal{\mathcal T}
\newcommand\then{\!\!\implies\!\!}
\newcommand\thens{\!\!\impliess\!\!}
\newcommand\tos{\buildrel * \over \to}
\newcommand\uno{\mathbf{1}}
\newcommand\wt{\widetilde}
\newcommand\V{\mathbf{V}}
\newcommand\VV{X}   
\newcommand\vv{\mathbf v} 
\newcommand\ww{\mathbf w} 
\newcommand\XX{X}   
\newcommand\yy{\mathbf y} 

\newcommand\zf{\mathfrak{z}}

\numberwithin{equation}{section}

\newtheoremstyle{mythm}
  {9pt}
  {9pt}
  {\itshape}
  {0pt}
  {\bfseries}
  {}
  { }
  {\thmnumber{(#2)}\thmname{ #1}\thmnote{ #3}}

\newtheoremstyle{mydef}
  {9pt}
  {9pt}
  {\normalfont}
  {0pt}
  {\bfseries}
  {}
  { }
  {\thmnumber{(#2)}\thmname{ #1}\thmnote{ #3}}

\theoremstyle{mythm}
\newtheorem{thm}[equation]{Theorem.}
\newtheorem{pro}[equation]{Proposition.}
\newtheorem{lem}[equation]{Lemma.}
\newtheorem{cor}[equation]{Corollary.}

\theoremstyle{mydef}
\newtheorem{dfn}[equation]{Definition.}

\newtheorem{ass}[equation]{Assumptions.}
\newtheorem{rmk}[equation]{Remark.}

\begin{document}$\,$ \vspace{-1truecm}
\title{Context-free pairs of groups \\ II - cuts, tree
sets, and random walks}
\author{\bf Wolfgang WOESS}
\address{\parbox{.8\linewidth}{Institut f\"ur Mathematische Strukturtheorie, 
\\ Technische Universit\"at Graz,\\
Steyrergasse 30, 8010 Graz, Austria\\}}
\email{woess@TUGraz.at}
\date{\today} 
\thanks{The author was supported by the
Austrian Science Fund project FWF-P19115-N18}

\begin{abstract} This is a continuation of the study, begun by {\sc
Ceccherini-Silberstein and Woess}~\cite{CeWo3}, of context-free pairs of groups
and the related context-free graphs in the sense of 
{\sc Muller and Schupp}~\cite{MS2}. Instead of the cones (connected 
components with respect to deletion of finite balls with respect to the graph
metric), a more general approach to context-free graphs is proposed
via tree sets consisting of cuts of the graph, and associated structure trees. 
The existence of tree sets with certain ``good'' properties is studied.
With a tree set, a natural context-free grammar is associated. These
investigations of the structure of context free pairs, resp. graphs are then
applied to study random walk asymptotics via complex analysis.
\end{abstract}

\maketitle

\markboth{{\sf W. Woess}}
{{\sf Context-free pairs of groups. II}}
\baselineskip 15pt

\section{Introduction and preliminaries}\label{sec:intro}

This is a direct continuation of the paper of 
{\sc Ceccherini and Woess}~\cite{CeWo3}. 

Let us briefly outline the setting. 
We have a finitely generated group $G$, a subgroup $K$ and a
a finite \emph{alphabet} $\Si$ together with a mapping $\psi: \Si \to G$
such that $A = \psi(\Si)$ generates $G$ as a semigroup. We call $\psi$ a 
\emph{semigroup presentation} of $G$.
We also write $\psi$ for the extension of this mapping as a monoid
homomorphism $\psi: \Si^* \to G$, where $\Si^*$ consists of all words over
$\Si$, with word concatenation as the semigroup product and the empty word 
$\epsilon$ as the unit element. 
The pair $(G,K)$ is called \emph{context-free} if the language
$
L(G,K,\psi) = \{ w \in \Si^* : \psi(w) \in K \}
$
is context-free. This is independent of the specific choices of
$\Si$ and $\psi$, see \cite[Lemma 3.1]{CeWo3}.

Let us recall here that a context-free language is a subset
of $\Si^*$ that is generated by a \emph{context-free grammar} 
$\Ccal = (\V,\Si,\Pb,S)$, 
where  $\V$ is the (finite) set of \emph{variables} (with 
$\V\cap \Si = \emptyset$), the 
variable $S$ is the \emph{start symbol,} and $\Pb \subset \V \times (\V \cup \Si)^*$ is a 
finite set of \emph{production rules.} We write $T \vdash u$ 
if $(T,u) \in \Pb$.  For $v, w \in (\V \cup \Si)^*$, a \emph{rightmost 
derivation step} has the form $v \then w$, where $v = v_1Tv_2$ and 
$w=v_1uv_2$ with 
$u, v_1 \in (\V \cup \Si)^*$, $v_2 \in \Si^*$ and $T \vdash u$. 
A \emph{rightmost derivation} is a sequence
 $v=w_0,w_1, \dots,w_k=w \in (\V \cup \Si)^*$ such 
that $w_{i-1} \then w_i\,$; we then write $v \thens w$. 
Each $T \in \V$ generates the language  
$L_T = \{ w \in \Si^* : T \thens w \}$. The \emph{language generated by} 
$\Ccal$ is $L(\Ccal) =L_S$. The grammar is called \emph{un-ambiguous,} if
every $w \in  L(\Ccal)$ has a unique rightmost derivation.

{\sc Harrison~\cite{Ha}} is an excellent source on context-free languages.

The group $G$ itself is called context-free, if the language
$L(G,\{ 1_G \},\psi)$ is context-free.
{\sc Muller and Schupp}~\cite{MS1} have shown that a group is context-free if
and only if it is virtually free. 
Subsequently, in \cite{MS2} they have introduced context-free \emph{graphs},
which play a crucial role in the present work. 
For additional references related to the subject of this paper, see also
\cite{CeWo3}. The graphs that we are dealing with are \emph{Schreier graphs}
of $(G,K)$. 

\smallskip

\emph{In the present paper, we shall always assume that the index of $K$ in $G$,
as well as the graphs in consideration, are infinite.}

\smallskip

Before proceeding with the introductory outline, let us include some
quick reminders regarding those graphs. 

\smallskip

Let $\Si$ be a finite alphabet. 
A \emph{directed graph $(\VV, \EE, \ell)$ labelled by $\Si$} consists of the 
(finite or countable) set of \emph{vertices,} the set of \emph{oriented, 
labelled edges} $\EE \subset \VV \times \Si \times \VV$ and the 
\emph{labelling} $\ell :\EE \ni (x,a,y) \mapsto a \in \Si$. 
Loops are allowed. Multiple edges between the same vertices must have
distinct labels. All our graphs will be \emph{fully deterministic:} for every vertex 
$x$ and label $a \in \Si$, there is precisely one edge with label $a$ starting
at $x$. Our graphs will usually also be \emph{symmetric (undirected):} 
there is a proper involution $a \mapsto a^{-1}$ of $\Si$ such that for
$e=(x,a,y) \in \EE$, also $e^{-1}=(y,a^{-1},x)$ belongs to $\EE$. 

A \emph{path} in $\XX$ is a sequence $\pi = e_1 e_2 \dots e_n$ of edges such 
that the terminal vertex of $e_i$ is the intial vertex of $e_{i+1}$ for 
$i=1,\ldots, n-1$. Its \emph{label} is 
$\ell(\pi) = \ell(e_1) \ell(e_2) \cdots \ell(e_n) \in \Si^*$. The \emph{empty 
path} starting and ending at $x$ has label $\epsilon$. 
We only consider graphs that are \emph{strongly connected:} for all $x,y \in
\XX$, there is a path from $x$ to $y$. 
With any two vertices $x,y \in \XX$, we associate the language of all words
that can be read along some path from $x$ to $y$, that is,
$L_{x,y} = \{ \ell(\pi) : \pi \;\text{a path from $x$ to $y$}\}$.

The \emph{Schreier graph} $\XX = \XX(G,K,\psi)$ of $(G,K)$ with respect to
$\psi$ has the vertex set $\VV=\{ Kg: g \in G \}$ of all right $K$-cosets 
in $G$, and the set of labelled, directed
edges 
$
\EE = \{e = (x,a,y): x = Kg\,,\; y = Kg\psi(a)\,,\;\text{where}\; 
g \in G\,,\; a \in \Si\}\,.
$
The vertex $o = K$ serves as a root of $\XX$. When $\psi$ is \emph{symmetric}, 
that is, there is a proper involution 
$a \mapsto a^{-1}$ of $\Si$ such that $\psi(a^{-1}) = \psi(a)^{-1}$ then the
Schreier graph is symmetric.

We have that $L(G,K,\psi)$ is context-free if and only if the language $L_{o,o}$
associated with $\XX(G,K,\psi)$ has this property. 

We describe the concept of context-free graphs. Let $\XX$ be a labelled,
symmetric graph with a chosen root vertex $o$. By symmetry, it has its natural 
graph metric $d$. Let $B(o,n) = \{ x \in \VV : d(x,o)\le n\}$ be the ball with 
radius $n$ centred at $o$. Then a \emph{cone} of $\XX$
with respect to $o$ is a connected component of $\XX \setminus B(o,n)$ with
$n \ge 0$. Each cone is a labelled graph with its boundary consisting of
all its vertices having a neighbour in the complement. 
Then $\XX$ is called a \emph{context-free graph} in \cite{MS2}, if there are 
only finitely many isomorphism types, as labelled graphs with boundary, 
of cones with respect to $o$.

It is shown in  \cite{CeWo3} that a fully deterministic, symmetric graph
$X$ with any chosen root vertex $o$ is a context-free graph in the above sense 
if and only if $L_{o,o}$ is a context-free language. (Contrary to what one 
tends to believe at first glance, this is not contained in  \cite{MS1} and
\cite{MS2}, the ``if'' part being the harder one.)

\smallskip

Let us come back to the introductory outline.
In the present paper, we first introduce a more general approach to context-free graphs,
via \emph{oriented tree sets} consisting of \emph{cuts}, and associated
\emph{structure trees} in the spirit of {\sc Dunwoody}~\cite{Du1}, \cite{Du2} and 
{\sc Dicks and Dunwoody}~\cite{DiDu}; see also 
{\sc Thomassen and Woess}~\cite{ThWo}. 
Our tree sets have \emph{finite type} (the cuts which they contain 
belong to finitely many isomorphism classes) and \emph{tesselate} the underlying 
graph  (see below for precise explanations). 
We first perform a detailed study of those tree sets in our graphs, 
showing that 
they can always be modified so that they have certain connectivity
and separation properties.  A crucial additional property, \emph{irreducibility,} is 
introduced. For Schreier graphs it is shown that irreducibility is preserved 
under finite index extensions of the group $G$. Also, 
the existence of a tree set with those properties is independent of the 
presentation map~$\psi$.
As a particular, interesting class of of examples, all those ``good''
properties, including irreducibility, hold when $G$ is virtually free and
$K$ is a finitely generated free subgroup.

With a tree set with finite type that tesselates the Schreier graph of $(G,K)$ 
with respect to $\psi$, we can associate
in a natural way an un-ambiguous context-free grammar that generates 
$L(G,K,\psi)$. The above properties of the tree set translate into 
properties of the grammar and its \emph{dependency digraph.}

In the final part, we change the flavour from structure-theoretic
considerations to random walks and the analysis of generating functions.
We start with a fully deterministic, symmetric graph and equip $\Si$ with a
probability measure $\mu$. This induces a random walk on $\XX$, where 
a random step from a vertex $x$ along an edge $(x,a,y)$ has probability
$\mu(a)$. Suppose that the graph has a tree set as above.
Via the fundamental theory of {\sc Chomsky and Schutzenberger}~\cite{ChSc},
the associated grammar translates into an algebraic system of equations for the
generating functions associated with the transition probabilities of 
the random walk on our context-free graph. The ``good'' properties of the tree
set, in particular irreducibility, guarantee that this  system of equations
allows to apply the methods of complex analysis
that in the monograph of {\sc Flajolet and Sedgewick}~\cite[\S VII.6]{FlSe}
are subsumed as the Drmota--Lalley--Woods theorem. As a matter of fact,
that theorem applies directly only to the generating functions of certain 
restricted transitions, while the final step regarding $n$-step return probabilities
$p^{(n)}(x,x)$ needs some additional care. As a result, we get the following
alternative between three possible cases of asymptotic behaviour for all
vertices $x$ of the context-free graph.
\begin{align}
p^{(n)}(x,x) &\sim c_{x,x} \,R_{\mu}^{-n} \,, \quad \text{or}\label{eq:asy1} 
\\[1pt] 
p^{(n)}(x,x) &\sim c_{x,x} \,R_{\mu}^{-n}\, n^{-1/2} \quad
\text{or}\label{eq:asy2} 
\\[1pt]
p^{(n)}(x,x) 
&\sim \bigl(c_{x,x} + (-1)^n \bar c_{x,x}\bigr)\,R_{\mu}^{-n}\, n^{-3/2} \,, 
\label{eq:asy3}\end{align} 
as $n \to \infty$ (with $n$ even when the graph has no odd cycles), 
where $0 \le |\bar c_{x,x}| < c_{x,x}$ and $R_{\mu} \ge 1$. The oscillating
term in (3) can occur when the graph does have odd cycles, while there are none
outside of some fixed finite set. The basic example exhibiting those
oscillations is reflecting random walk on the non-negative integers,
as studied by {\sc Lalley}~\cite{La2}. 

In terms of the underlying groups $G$ and $K$, we can just look at the random 
walk 
induced by $\mu$ on $G$ itself. Then $p^{(n)}(o,o)$ is the probability that
this walk, starting at $1_G$, is in $K$ at the $n$-th step.

In particular, when $K = \{ 1_G\}$ and $G$ is virtually free, but not virtually
cyclic, then we obtain the asymptotic behaviour (3), without oscillating term:
\begin{equation}\label{eq:loc-virtfree}
p^{(n)}(x,x) \sim c_{x,x} \,R_{\mu}^{-n} \, n^{-3/2}\,,
\end{equation}
as $n \to \infty$ (and $n$ is even when the underlying Cayley graph of $G$ has
no odd cycles) where $R_{\mu} > 1$ since the group is non-amenable.
Note the small restriction that we require that the support of the
measure $\mu$, or more precisely, the set $A = \psi(\Si)$ is a symmetric
subset of $G$. Up to this (which can be amended by additional work), we get 
the general local limit theorem for finite range random walks on virtually 
free groups.

Behind this result, there is a long history. See the survey of 
{\sc Woess}~\cite{Wsurv} 
for quite complete references to work regarding asymptotics of random walks 
on free groups. In the present context of finite range random walks, the
most significant contribution is the one of {\sc Lalley}~\cite{La1} regarding
free groups.
The use of context-free languages to derive infomation about random walk
asymptotics on virtually free groups appears first in {\sc Woess}~\cite{Wcf}.
This is closely related to the random walks on regular languages
of \cite{La2}, which provide another possible approach to random walks
on virtually free groups. 
 
In addition to \eqref{eq:loc-virtfree}, we also obtain the local limit theorem
for a larger class of Schreier graphs, resp. context-free graphs, and undertake
a careful study of the behaviour in the periodic case, when the oscillations
in \eqref{eq:asy3} can occur. As a matter of fact, when dealing with Schreier
graphs instead of just groups, such a situation is quite natural.

\smallskip

\noindent{\bf Acknowledgement.} The author is grateful to 
Tullio Ceccherini-Silberstein for
many discussions on this work, and to Laurent Bartholdi 
and Brian Bowditch for providing the references \cite{Bog} and \cite{Sc},
respectivley, in the context of
Theorem \ref{thm:virtfree}.


\section{Cuts and structure trees of context-free graphs}\label{sec:cuts}

For the study of certain aspects of context-free groups, pairs and graphs,
it will be preferable to use more general types of connected subsets than
the cones with respect to some root vertex. In the sequel, we shall among other
be  interested in the property that for every pair of cones $C, D$, there is 
a cone  $C' \subset C$ which is isomorphic with $D$. 
Now while it may be feasible -- for example by using translation by
a suitable group element --
to show that $C$ contains a subgraph isomorphic with $D$, it is in general
not so clear how to obtain such a copy of $D$  that is again a cone with 
respect to deleting a ball $B(o,n)$ for the same root vertex $o$ as before.
This is also so if one replaces -- as in \cite{CeWo3} -- the root vertex with
a finite reference set $F$ for defining the cones as the components left when 
deleting the balls $B(F,n)$ ($n \ge 0$).

Therefore we now introduce the setting for a more general construction of 
context-free graphs.

We start with an infinite symmetric, fully deterministic labelled 
graph $(\VV,\EE,\ell)$.
A \emph{cut} is an infinite, connected induced subgraph $C$ of $\XX$ such that
$\bd C$ is finite and non-empty. (The boundary $\bd C$ consists of all
vertices in $C$ that have a neighbour outside $C$.)
Thus, there are only finitely many edges between $C$ and  its complement
$C^*$.
With a slight deviation from the terminology of \cite{Du2},
\cite{DiDu} and \cite{ThWo} , two cuts $C$ and $D$ are 
said to be \emph{non-crossing} if one of
\begin{equation}\label{eq:cross}
C \subset D\,,\;D \subset C\,,\;\; \text{or} \;\;C \cap D = \emptyset
\end{equation}
holds. (In the original definition, also the option $C^* \subset D$ is
included.)

An \emph{oriented tree set} in $X$ is a non-empty family $\Ecal$
of pairwise non-crossing cuts in $X$. We also require that
\begin{equation}\label{eq:Max}
M = \max \{ \diam(\partial C) : C \in \Ecal \} < \infty\,.
\end{equation}

\begin{lem}\label{lem:complement}
There is at most one pair of cuts $C, C^*$ such that both belong to $\Ecal$,
and then both $C$ and $C^*$ must be maximal in $\Ecal$ with respect to set 
inclusion.
\end{lem} 

\begin{proof} Let $C, C^*$ be as stated, and let $D \in \Ecal$ be another cut.

If $D \subsetneq C$ then we cannot have $C \subset D^*$ 
(since otherwise $X = C \cup C^* \subset D^*$), nor $D^* \subset C$,
nor $D^* \cap C = \emptyset$. Therefore $D^* \notin \Ecal$.

Similarly, one shows that if $D$ is a cut with $C \subsetneq D$ then
$D \notin \Ecal$.  
\end{proof}

We mention two related properties of $\Ecal$ which are not hard to prove;
see e.g. \cite{Du2}.
\begin{equation}\label{eq:finprop}
\begin{aligned} 
&\text{If $C \in \Ecal$ and $x \in C$, then 
$\;\{ C' \in \Ecal: x \in C' \subset C \}\;$ is finite.}\\
&\text{If $C,D \in \Ecal$ and $D \subset C$, then 
$\;\{ C' \in \Ecal: D \subset C' \subset C \}\;$ is finite.}\\
\end{aligned}
\end{equation}

For $C, D \in \Ecal$, we say that $D$ is a \emph{successor} of $C$, notation
$C \to D$, if there is no $C' \in \Ecal$ such that 
$C \supsetneq C' \supsetneq D$. In that case, $C$ is of course
a \emph{predecessor} of $D$, notation $C = D^-$. 

\begin{lem}\label{lem:tree} One of the following holds:
\\[5pt]
{\rm (1)} Every element of $\Ecal$ is contained in a cut in $\Ecal$ that is
maximal with respect to set inclusion, or
\\[5pt]
{\rm (2)} There is a strictly increasing sequence $(C_k)$ in $\Ecal$ whose union
is $X$, and for every $C \in \Ecal$, there is $k$ such $C \subset C_k\,$.
\end{lem}

\begin{proof} Suppose that there is some element in $\Ecal$ that is not
contained in a maximal cut in $\Ecal$. Then there must be a strictly increasing 
sequence $(C_k)_{k \ge 0}$ in $\Ecal$. Let $x \in X$ be arbitrary.
Then $d(x,\partial C_k) \to \infty$ as $k \to \infty$.
Let $\pi$ be a path that connects $x$ with some element of $C_0$, and 
set $n = |\pi|$. If $k$ is such that $d(x,\partial C_k) > n$ then
$\pi$ is a connected subset of $X$ that does not intersect $\partial C_k$.
Therefore we must have $\pi \subset \partial C_k^*$ or  
$\pi \subset \partial C_k$. Since $\pi$ intersects $C_0 \subset C_k$, the 
second case must be true. In particular, $x \in C_k$.

We see that $X = \bigcup_k C_k$. Now let $C \in \Ecal$ be arbitrary.
Then there must be $k_0$ such that $C \cap C_k \ne \emptyset$ for all
$k \ge k_0\,$. We cannot have $C_k \subset C$ for all $k \ge k_0\,$, 
since otherwise $C = X$. If $k \ge k_0$ is such that $C_k \not\subset C$, 
then of the three cases of \eqref{eq:cross}, only $C \subset C_k$ remains.
\end{proof}

Following \cite{Du1} and \cite{DiDu}, we construct
the \emph{(oriented) structure tree} $\Tcal=\Tcal_{\Ecal}\,$. 
There is a slight difference, in that edges are only oriented in
one way: the (oriented, non-symmetric) edge set is $\Ecal$. 
The tree is obtained as follows: with each $C \in \Ecal$, we associate
its terminal vertex $\xi = \xi_C\,$. It coincides with the initial
vertex in $\Tcal$  of $D \in \Ecal$ precisley when $C \to D$.

Furthermore, if $\Ecal$ contains maximal elements, then we introduce
one vertex $\xi_0$  as the initial vertex of each maximal cut in
$\Ecal$ (as an edge of $\Tcal$).
 
It is clear that this defines an oriented tree. In case (1) of Lemma 
\ref{lem:tree}, we say that $\Ecal$ is \emph{rooted}. The tree has the root 
vertex $\xi_0\,$, and the edges
are oriented away from that root. In case (2), there is no such root,
but there is an end of the tree such that all edges point away from that end. 
Since we shall restrict our attention to case (1), we omit giving further 
details on the other case. We let 
$$
\partial_{\Ecal} \XX =  \bigcap\{ C^* : C \in \Ecal\}\,.
$$
In the rooted case, this \emph{root set} is the intersection of the complements of the
maximal elements of $\Ecal$. Otherwise, in case (2), $\partial_{\Ecal} \XX$ is 
empty. For $C \in \Ecal$, let
$$
\partial_{\Ecal} C = C \setminus \bigcup \{D \in \Ecal : D^-=C\}\,.
$$
Thus, $X$ is the disjoint union of $\partial_{\Ecal} X$ and all 
$\partial_{\Ecal} C\,$, where $C \in \Ecal$.
With this partition, we can associate the \emph{structure map} 
$x \mapsto \xi(x)$ from $X$ to $\Tcal_{\Ecal}\,$. If $x \in \partial_{\Ecal} X$
then $\xi(x) = \xi_0$. If $x \in \partial_{\Ecal} C$ then $\xi(x) = \xi_C\,$,
that is, $C$ is the unique minimal cut in $\Ecal$ that contains $x$.
(If the vertices $x, y \in X$ are neighbours, then $\xi(x)$ and $\xi(y)$
either coincide or are neighbours in $\Tcal_{\Ecal}\,$.)

We say that a rooted tree set $\Ecal$ \emph{tesselates} $\XX$, if 
$\partial_{\Ecal} \XX$ and  $\partial_{\Ecal} C$ are finite for every 
$C \in \Ecal$. In that case, the structure tree is locally finite, that is, for
every $C \in \Ecal$, there are only finitely many $D \in \Ecal$ with $C \to D$. 

Given $\rho \in \N$, we say that $\Ecal$ is \emph{$\rho$-separated}, 
if $\partial_{\Ecal} \XX \ne \emptyset$
and $d(\partial C, \partial D) \ge \rho$ whenever $C, D \in\Ecal$ are 
distinct. When $\rho=1$, we just say ``separated''.
The collection of all cones with respect to a root vertex $o$
is a rooted tree set that tesselates $\XX$ and is separated. 

\begin{dfn}\label{def:fintype}
Let $X$ and $\Ecal$ (rooted, oriented) be as above.
\\[5pt]
(a) We say that $\Ecal$ has \emph{finite type}, if it tesselates $\XX$ and
there are elements 
$$
C_1\,,  \dots, C_r \in \Ecal
$$
such that every cut $C \in \Ecal$ is isomorphic with one of the 
$C_i$ as a labelled graph, and that isomorphism (which is chosen and fixed 
for each $C$) maps $\partial_{\Ecal} C$ to 
$\partial_{\Ecal} C_i$. In that case, we say that $C$ has \emph{type} $i$. 
\\[5pt]
(b) We say that $\Ecal$ is \emph{irreducible} if it has finite type and 
the cuts $C_1\,, \dots, C_r$ of (a)   
are such that for all $i, j$, the cut $C_i$ contains as a proper subset 
a cut in $\Ecal$ that is isomorphic with $C_j\,$.
\end{dfn}

In a context-free graph $\XX$, the cones with respect to a root vertex $o$, or 
with respect to a connected, finite set $F$ as in \cite{CeWo3}, 
are a special case of (a).

Note that in case of finite type as in (a), the isomorphism from $C$
to $C_i$ must also map $\partial C$ to $\partial C_i\,$, since the
boundaries consist precisely of those points where the respective cut is
not fully deterministic. Furthermore, since $\partial_{\Ecal} C$ is mapped onto 
$\partial_{\Ecal} C_i\,$,
the number of $D \in \Ecal$ with type $j$ for which $C \to D$, resp. 
$C_i \to D$ is the same.

If (a) holds, the structure tree
with root $\xi_0$ is a tree with finitely many cone types in the sense of 
{\sc Nagnibeda and Woess}~\cite{NaWo}; see also \cite[\S 5]{CeWo2}.
If $C \in\Ecal$ is isomorphic with $C_i$ then
we say that the associated edge of $\Tcal$, as well as its terminal
vertex in the tree, have \emph{type} $i$. The root $\xi_0$ has type $0$.

For irreducibility it is necessary that each $C \in \Ecal$ is infinite.

Similarly as in the proof of Theorem 4.2 in \cite{CeWo3}, we can
encode the type stucture in a finite, oriented \emph{graph of types} 
$\Gamma$ over the vertex set $I= \{ 1,\dots, r\}$, where we have $a(i,j)$ 
oriented edges $i \to j$ whenever there are precisely $a(i,j)$ cuts 
$D \in \Ecal$ such that $C_i \to D$ and
$D$ has type $j$. Ths means that we enumerate the successor cuts with type $j$
of $C_i$ by the numbers $1, \dots, a(i,j)$, and this enumeration is carried
over to any cut $C$ with type $i$ and its successors with type $j$ via the
isomorphism $C \to C_i$. 

For our use, it will be good to have a separated tree set.

\begin{pro}\label{pro:goodcuts}
Let $\Ecal$ be a rooted, oriented tree set consisting of
cuts of the fully deterministic, symmetric labelled  graph $\XX$.
\\[5pt]
\emph{(1)} Suppose that $\Ecal$ has finite type.
Then for every $\rho \in \N$,
there is a $\rho$-separated rooted tree set $\ol\Ecal \subset \Ecal$ 
that tesselates $\XX$ and such that $\partial_{\ol\Ecal} \VV \supset F$. 
(It has finite type.)
\\[5pt] 
\emph{(2)}
If in addition $\Ecal$ is irreducible, then $\ol\Ecal$ can be constructed
such that it is also irreducible.
\\[5pt] 
\emph{(3)} If $F \subset X$ is finite, then $\Ecal$ can be modified
such that $\partial_{\Ecal} \XX \supset F$. 
In addition, the cuts $C$ in $\Ecal$ can be modified such that
all $\partial_{\Ecal}C$ as well as $\partial_{\Ecal}\XX$ are connected. These 
modifications
preserve $\rho$-separation, the property that $\partial_{\Ecal}\XX \subset F$,
and irreducibility, respectively. 
\end{pro} 

\begin{proof}
First of all, there is a finite bound on the number of elements in 
$\partial C_i\,$, $i=1, \dots,r$. By \eqref{eq:finprop},
the number
$$
s = \max \Bigl\{\bigl| \{ D \in \Ecal : C_i \supsetneq D \ni x \}\bigr| : 
x \in \partial C_i\,,\;i \in I \Bigr\}
$$
is finite. When $C, D \in \Ecal$ are such that as
edges of $\Tcal$, their initial points have distance at least $s$, then
their boundaries do not intersect. If the initial points have distance
at least $\rho s$ in $\Tcal$, then their boundaries have distance at least
$\rho$ in the graph $X$. 

We now choose an integer $k \ge \rho s$ and modify $\Ecal$ as follows.

The new tree set $\overline\Ecal$ consists of all $C$ in $\Ecal$ whose 
terminal vertex
(as an edge of $\Tcal$) is at distance $m k$ from $\xi_0\,$, where 
$m \ge 1$ (integer). 
This is a new tree set. The associated structure tree $\overline\Tcal$ is 
obtained from $\Tcal$ by keeping only the vertices of $\Tcal$ at distance
$m k$ from $\xi_0$, where $m\ge 0$. Each of them is connected by an edge 
to all its descendants which are at distance $(m+1)k$ from $\xi_0\,$.
Here, by a descendant of a vertex $\xi$ of $\Tcal$, we mean another
vertex $\eta$ with the property that $\xi$ lies on the geodesic path from
$\xi_0$ to $\eta$.

Since $\Ecal$ tesselates $\XX$ and $\Tcal$ is locally finite, also
$\overline\Ecal$ tesselates $\XX$. There are only finitely many isomorphism
types in $\overline\Ecal$. 
Thus, $\overline\Ecal$ is as required by statement (1).

\smallskip

\noindent
(2) If the original $\Ecal$ is irreducible,
then it is not immediate that the above $\overline\Ecal$ inherits that
property, and some additional effort is needed.

Consider the adjacency matrix $A = \bigl(a(i,j)\bigr)_{i,j\in I}$ of the 
graph of types $\Gamma$ associated  with $\Ecal$. 
Let $A^n = \bigl(a^{(n)}(i,j)\bigr)_{i,j \in I}$ be the $n$-th matrix power.
The assumption of statement (2) says that $A$ is an irreducible matrix:
for all $i, j$ there is $n=n_{i,j}$ such that $a^{(n)}(i,j) > 0$. By the 
elementary theory of irreducible non-negative matrices (see e.g. 
{\sc Seneta}~\cite{Se}), there are positive integers $\dsf$ (the 
\emph{period} of $A$) and $n_0$ with the following properties.  

\smallskip

$\bullet\,$ The index set has a partition $I = I_0 \cup \dots \cup I_{\dsf-1}$
such that if $i \in I_k$ and $a(i,j) > 0$, then $j \in I_{k+1}\,$, where
we set $I_{\dsf}=I_0\,$.

$\bullet\,$ $a^{(n\dsf)}(i,j) > 0$ for all $n \ge n_0$
and all $i,j$ that belong to the same block.

\smallskip

We choose $\bar n \ge n_0$ such that $\bar \dsf = \bar n \dsf > \rho s$.
For each $i$, we can find $n_i > \rho s$  such that $a^{(n_i)}(i,j)> 0$
for some $j \in I_0\,$. Note that then we must have $j' \in I_0$
whenever $a^{(n_i)}(i,j')> 0$.  

We now describe how to modify $\Tcal$ to obtain the reduced tree 
$\overline\Tcal$. Let $\xi$ be any neighbour of $\xi_0\,$, and suppose that $\xi$
has type $i$. Consider the branch of $\Tcal$ that starts with $\xi$, that is,
the subtree spanned by all descendants of $\xi$. Of that branch, we keep all 
those vertices (and their types) that are at distance $n_i + m\bar \dsf$ from 
$\xi$, where $m \ge 0$. 

We now connect $\xi_0$ by a new oriented edge with those elements of the branch 
that are at distance $n_i$ from $\xi$, and we connect each of the 
vertices that we have kept in our branch with each of its descendants 
at distance $d$ in the original~$\Tcal$. 
We do this for every neighbour $\xi$ of $\xi_0\,$. We obtain a new
tree $\overline \Tcal$  whose vertices as well as their types are also
part of the original $\Tcal$. 
By construction, all those vertices have their type in $I_0$.

The corresponding tree set $\overline\Ecal$ consists of all $C \in \Ecal$
whose terminal vertex in $\Tcal$ belongs to $\overline\Tcal$.
By construction, $\overline\Ecal$ tesselates $X$, is separated,
and has finite type. Let $C \in \overline\Ecal$ have type $i \in I_0\,$.
If $j \in I_0$ then $a^{(\bar\dsf)}(i,j) > 0$. This means that in 
the orginial $\Tcal$, the cut $C$ (as an edge) 
has a descendant $D$ of type $j$ at distance $\bar\dsf$ (that is, the endvertices
of those edges of $\Tcal$ have distance $\bar\dsf$). But then $D \in \ol\Ecal$. 

We now see that $\ol\Ecal$ is irreducible, and its types are given by the
set $I_0\,$.
\smallskip

\noindent
(3) We first consider the finite set $F$ and the structure tree 
$\Tcal = \Tcal_{\Ecal}\,$.  Recall the structure map 
$x \mapsto \xi(x)$ from $X$ to $\Tcal_{\Ecal}$ that we have described after
introducing the structure tree. 
We delete from $\Ecal$ every cut $C$ which, as 
an oriented edge of $\Tcal$, lies on the geodesic path from
$\xi_0$ to $\xi(x)$ for some $x \in F$. Only finitely many $C$ are
deleted. We obtain a new, smaller tree set $\overline\Ecal$. The associated structure
tree $\overline\Tcal$ is obtained from $\Tcal$ by contracting those 
geodesic paths to $\xi_0\,$, which we also consider as our root vertex
of $\overline\Tcal$. Then we get with respect to the new tree set
$$
\partial_{\ol\Ecal} C = \partial_{\Ecal} C\,,\quad\text{if}\,\; C \in  
\overline\Ecal\,, \AND
\partial_{\ol\Ecal} X 
= \partial_{\Ecal} X \cup \bigcup \bigl\{ \partial_{\Ecal} C : 
C \in \Ecal \setminus \ol\Ecal \bigr\}\,.
$$
Thus, $\ol\Ecal$ inherits from $\Ecal$ all properties stated in the
proposition,  and $F \subset \partial_{\ol\Ecal}X$. 

\smallskip

So now let us assume that already $\Ecal$ itself is such that
$F \subset \partial_{\Ecal}\XX$, and 
that $\Ecal$ is $\rho$-separated, where $\rho \ge M$, with
$M$ as in \eqref{eq:Max}.

For each $C \in \Ecal$, we let 
$\wt C = \{ x \in C : d(x, \partial C) \ge M \}$. Then
$\partial \wt C = \{ x \in C : d(x, \partial C) = M \}$. The subgraph
$\wt C$ of $X$ is not necessarily connected, but -- having
finite boundary -- it will fall apart into infinite connected
components $\wt C^{(1)}, \dots, \wt C^{(k)}$, where $k=k(C)$ depends on $C$,
plus maybe a finite number of finite components.
Since $\phi$ is an isomorphism from $C$ to $C_i$ ($i \in \{1, \dots, r\}$),
it maps each $\wt C^{(l)}$ to  $\wt C_i^{(l)}$ ($l=1, ..., k$) when those
components are numbered accordingly.

Let $C, D \in \Ecal$ with $C \to D$, and consider $\wt D^{(l')}$, where
$l' \in \{1, \dots, k(D)\}$. This is a connected subgraph of $\XX$ which
does not intersect $\partial \wt C$, while it does intersect $\wt C$.
Therefore it must be contained in one of the components of $\wt C$, that
is, in some $\wt C_i^{(l)}$ ($l \in \{ 1, ..., k(C)\}$).
We infer that
$$
\wt \Ecal = \{ \wt C^{(l)} : C \in \Ecal\,,\; l = 1,\dots, k(C) \}
$$
is a $\rho$-separated, oriented tree set with finite type that tesselates $\XX$.

Next, suppose that $\Ecal$ is irreducible. Consider $\wt C_i^{(l)}$,
where $l \le k(C_i)$, and let $j \in \{ 1, \dots, r\}$. Since  
$\wt C_i^{(l)}$ is an infinite subgraph of $X$, it contains some element
$x$  that is sufficiently far from $\partial\wt C_i^{(l)}$ so that it is contained
in some $D \in \Ecal$ with $d(\partial D, \partial \wt C_i^{(l)}) > 0$. 
But then we must have $D \subset \wt C_i^{(l)}$. Now $D$ contains some 
$C \in \Ecal$ that is isomorphic with $C_j$. In turn, $C$ contains
$\wt C^{(l')}$ for all $l' \in \{1, \dots, k(C)\}$, which
are (respectively) isomorphic with $\wt C_j^{(l')}$ for 
$l' \in \{1, \dots, k(C)\}$. Thus, $\wt \Ecal$ is again irreducible.

Finally, we show connectedness of $\partial_{\wt\Ecal} \wt C^{(l)} \in \wt \Ecal$, 
where $C \in \Ecal$.
We fix some point $x_0$ in $\partial \wt C^{(l)}$ and show that every point 
$x \in \partial_{\wt\Ecal} \wt C^{(l)}$ is connected to $x_0$ by a path that lies
entirely within $\partial_{\wt\Ecal} \wt C^{(l)}$. 

Note first that -- as we have seen above -- when 
$\wt C^{(l)} \to \wt D^{(l')} \in \wt\Ecal$, where $D \in \Ecal$, 
then   $\wt C^{(l)} \supset \wt D$. 

(i) If $x \in D$ for such a cut $D \in \Ecal$, then we have the following
possibilities. 
Either $x \in D_M= \{ y \in D: d(x,\partial D) < M \}$, in which case 
$x$ is connected to some $y \in \partial D$ by a path that lies 
entirely within $D_M \subset \partial_{\wt\Ecal} \wt C^{(l)}$. 
Or else $x$ belongs to a finite connected component $F$ 
of $\wt D$ (otherwise $x \notin \partial_{\wt\Ecal} \wt C^{(l)}$). 
Then $x$  is connected by a path within $F \subset \partial_{\wt\Ecal} \wt C^{(l)}$ 
with some element of $\partial F \subset \partial \wt D$. The latter is 
connected to some element $y$ of $\partial D$ by a path of length $M$ that
lies within $D_M \cup \partial F$. 

(ii) If $x \in \partial_{\Ecal} C$ (with $d(x,\partial C) \ge M$)
then $x$ is connected to $x_0$ by some path that lies in 
$\wt C^{(l)}$. If that path does not exit from $\partial_{\Ecal} C$ then
it lies entirely within $\partial_{\wt\Ecal} \wt C^{(l)}$, and we are done.
Otherwise, it exits $\partial_{\Ecal}C$ an enters some $D$ with $C \to D \in \Ecal$
and $\wt C^{(l)} \supset D$, as above. That is, $x$ is connected to some
$y \in \partial D$ by a path within  $\partial_{\Ecal}C$. 

(iii) We are left with having to show that when $D \in \Ecal$ is as above,
and $y \in \partial D$, then $y$ is connected to $x_0$ by a path within
$\partial_{\wt\Ecal} \wt C^{(l)}$. Now there is some path $\pi$ from $y$ to $x_0$
that lies within $\wt C^{(l)}$. This path may go back and forth 
between $\partial_{\Ecal} C \cap \wt C^{(l)}$ and some $D$ as above several
times, and we use induction on the number of those crossings. 
If $\pi$ goes from $y$ directly into $\partial_{\Ecal} C \cap \wt C^{(l)}$
without hitting any further point of $D$, then it lies within 
$\partial_{\Ecal} \wt C^{(l)}$ as required. Otherwise, it makes one or more 
detours into  some $D'\in \Ecal$ with $C \to D'$ and $D' \subset \wt C^{(l)}$. 
(Initially, this may be our $D$ for which $y \in \partial D$.) 
At each of those detours, $\pi$ enters and exits $D'$ at points 
$y', y'' \in \partial D'$, respectively.
But $\diam(\partial D') \le M$, so that $y'$ and $y''$ are connected
by a path that lies within $D'_M  \subset  \partial_{\wt\Ecal} \wt C^{(l)}$.
Thus, we may replace each of those detours with such a path within
$\partial_{\wt\Ecal} \wt C^{(l)}$, and in the end we get another path from 
$y$ to $x_0$ that lies within $\partial_{\Ecal} \wt C^{(l)}$, as required.


We should also show separately that with respect to 
$\wt \Ecal$, the finite set $\partial_{\wt \Ecal} \XX$ is connected. 
This follows the same reasoning as above, in particular (iii), and is omitted.
\end{proof}

\begin{pro}\label{pro:inherit1} 
Let $G$ be finitely generated, $H$ 
a subgroup with $[G:H] < \infty$, and $K$ a subgroup of $H$.
Furthermore, let $\psi: \Si \to H$ and $\psi': \Si' \to G$ be 
symmetric semigroup presentations of $H$ and $G$, respectively.

Suppose that the Schreier graph $\XX=\XX(H,K,\psi)$ has a rooted, oriented 
tree set $\Ecal$ that is irreducible. Then 
also $\XX' = \XX(G,K,\psi')$ has a tree set $\Ecal'$ with the same properties.
\end{pro} 

\begin{proof} 
Let $d(\cdot,\cdot)$ be the graph metric of $\XX = \XX(H,K,\psi)$ and 
$d'(\cdot,\cdot)$ the graph metric of $\XX'=\XX(G,K,\psi')$.
The common ``root vertex'' of each of our Schreier graphs 
is $o=K$.

We start by observing that there are integers $M_1\,, M_2 \ge 1, N \ge 0$ that
depend only on $(H,\psi)$ and $(G,\psi')$, 
such that the following properties hold.
\begin{equation}\label{eq:metrics}
\begin{gathered}
\text{For every}\; x \in X'\;\text{there is}\;y \in X\;\text{ with} \;  
d'(x,y) \le N\,, \AND\\
\text{for all }\;y_1\,, y_2 \in X\,,\;\text{one has }\;  
d'(y_1,y_2)/M_1 \le d(y_1,y_2) \le M_2\,d'(y_1,y_2)\,.
\end{gathered} 
\end{equation}
For the sake of clarity, we subdivide the proof in two steps.

\smallskip

\noindent
\emph{Step 1.} \;$H=G$. 
Then $\XX$ and $\XX'$ have the same vertex set,
but different labelled edges which come from the two different symmetric
semigroup presentations $\psi, \psi'$ with associated
alphabets $\Si$ and $\Si'$, respectively. In this case, $N=0$.

We can assume that $\Ecal$ is $R$-separated, where $R=2M_2$. 

Let $C_i\,$, $i \in I$, be the finitely many representatitves
of the isomorphism types of $\Ecal$. 

For any $C \in \Ecal$, we now let 
$\wt C = \{ x \in C : d(x,\partial C) \ge M_2/2 \}$,
where $\partial C$ is the boundary with respect to the graph
structure of $X$. We now consider $\wt C$ as an induced subgraph of $X'$,
that is, with those labelled edges that it inherits from $X'$.
 
Assume that $C$ has type $i$, and let $\phi: C \to C_i$ be the
corresponding isomorphism between the two as labelled subgraphs of $\XX$.
We claim that the restriction $\wt \phi$ of $\phi$ to $\wt C$ is an 
isomorphism between $\wt C$ and $\wt C_i$ as labelled subgraphs of $\XX'$.
It is clear that it is a bijection. Let $x,x' \in \wt C$ be connected by the
edge $(x,b,x')$ of $\XX'$, where $b \in \Si'$. Write 
$\bar x=\phi(x)$ and $\bar x'=\phi(x')$. There is a word 
$v_b \in \Si^+$ with $|v_b| \le M_2$ such
that $\psi(v_b) = \psi'(b)$ in $G$. The unique path\footnote{It is here that we use crucially the assumption
that $\XX$ is fully deterministic: in this case, for every vertex $x$
and every word $v \in \Si^*$, there is a unique path $\pi_x(v)$ that starts at 
$x$ and has label $v$.} 
 $\pi_x(v_b)$ in $\XX$
that starts at $x$ and has label $v_b$ has length at most $M_2$, so that it
lies entirely within $C$. 
Therefore $\phi$ maps this path to the path $\pi_{\bar x}(v_b)$ which ends at
$\bar x'$. But then we have the edge $(\bar x,b,\bar x')$ in the 
edge set of $X'$.
This may become clearer in terms of cosets. We can write $\bar x=Kg$,
 $\bar x'=Kg'$ with $g,g' \in G$. If $\pi_{\bar x}(v_b)$ ends in $\bar x'$ then 
$Kg\psi'(b)=Kg\psi(v_b) = Kg'$.

Thus, $\wt\phi: \wt C \to \wt C_i$ is indeed an isomorphism. 

Let $(x,b,x')$ be again an edge of $X'$. Suppose that $x \in \wt C$ and
$d(x,\partial C) \ge 3M_2/2$. Then $d(x',x) \le M_2$, whence 
$d(x',\partial C) \ge M_2/2$. Therefore $x' \in C$. We conclude that
the boundary $\partial'\wt C$ of $\wt C$ in $X'$ is contained in the finite
set $\{x \in C : M_2/2 \le d(x,\partial C) < 3M_2/2\}$.

We continue similarly as in the proof of Proposition \ref{pro:goodcuts}(3).
The subgraph $\wt C$ of $X'$ is not necessarily connected, but -- having
finite boundary -- it will fall apart into infinite connected
components $\wt C^{(1)}, \dots, \wt C^{(k)}$, where $k=k(C)$ depends on $C$,
plus maybe a finite number of finite components.
Since $\wt\phi$ is an isomorphism from $\wt C$ to $\wt C_i\,$, 
it maps each $\wt C^{(l)}$ to  $\wt C_i^{(l)}$ ($l=1, ..., k$) when those
components are numbered accordingly.

Suppose that $C, D \in \Ecal$ with $C \to D$. Let $\wt D^{(l')}$ be one
of the components of $\wt D$ according to the above construction.
This is a connected subgraph of $\wt C$. Therefore it is contained
in one of the components of $\wt C$. 

At this point, we see that 
$$
\Ecal' = \{ \wt C^{(l)} : C \in \Ecal\,,\; 1 \le l \le k(C) \}
$$    
is a rooted, oriented tree set of $X'$ that has finite type.
 
Now suppose that $\Ecal$ is irreducible. 
We show that also $\Ecal'$ is irreducible. Consider $\wt C_i^{(l)}$. Since it
is infinite, we can choose $x \in \wt C_i^{(l)}$ sufficiently far from 
$\partial'\wt C_i^{(l)}$ such that it is contained in some $C \in \Ecal$ with 
$d(C, \partial'\wt C_i^{(l)}) \ge d(\partial C, \partial'\wt C_i^{(l)}) 
> M_1M_2$.
Then,  by \eqref{eq:metrics}, $d'(C, \partial'\wt C_i^{(l)}) > M_1$. Thus, the
set $C^{M_1} = \{ x \in \XX' : d'(x,C) \le M_1 \}$ induces a connected
subgraph of $X'$ that does not intersect $\partial'\wt C_i^{(l)}$, while it 
does intersect $\wt C_i^{(l)}$. Therefore it must be contained in $D$. In 
particular, $\wt C_i^{(l)} \supset C$ as sets. Given $C_j$ with $j \in I$, 
by assumption there is
$D \in \Ecal$ such that $C \supsetneq D$, and $D$ is isomorphic
with $C_j$ as a labelled subgraph of $\XX$. By construction, $D$
contains each set $\wt D_j^{l'}$ ($l'=1, \dots, k(D)$). Above, we have shown
that via the isomorphism $D \to C_j$ in $X$, the new cut $\wt D_j^{(l')}$ 
is isomorphic with $\wt C_j^{(l')}$ as a labelled subgraph of $\XX'$.
Therefore $\Ecal'$ is irreducible.
     
\smallskip
\noindent\emph{Step 2.} We now assume that $\XX=\XX(H,K,\psi)$ 
is as stated in the proposition. By Step 1, we only need to show that 
$\XX'=\XX(G,K,\psi')$ has a tree set with the required properties 
for \emph{some} symmetric semigroup presentation $\psi'$ of $G$.
Our choice of $\psi'$ is as follows. We let $g_0=1_G,g_1,\dots, g_m$
be representatives of the right cosets $Hg$ ($g \in G$) of $H$ in $G$.

Then we define $\Si' = \Si \uplus \{ b_1, b_1^{-1}, \dots, b_m, b_m^{-1}\}$ 
and set $\psi'(a) = \psi(a)$, if $a \in \Si$ and 
$\psi'(b_i^{\pm 1}) = g_i^{\pm 1}$ ($i=1, \dots, m$).  
This is a symmetric semigroup presentation of $G$. For $X$ and $X'$, we have
$N=0$ and $M_1=1$ in \eqref{eq:metrics}. We write $E$ and $E'$ for the
edge sets of $X$ and $X'$, respectively.

We work with a tree set $\Ecal$ that satisfies the same
assumptions as in Step 1, and is $\rho$-separated with $\rho=6M_2\,$. 

For $C \in \Ecal$, we first define 
$\breve C = \{ x \in C : d(x,\partial C) \ge 3M_2/2 \}$, and let
$$
\wt C = \biguplus_{k=0}^m \{ Khg_k : Kh \in \breve C \} 
= \breve C \cup \biguplus_{k=1}^m 
\{ x \in X' : (y,b_k,x) \in E' \;\text{for some}\; y \in \breve C \}.
$$  
Now suppose again that $\phi: C \to C_i$ is an isomorphism of those
cuts as labelled subgraphs of $X$. We construct $\wt\phi: \wt C \to \wt C_i$                   
in the only possible way. Namely, $\wt\phi(x) = \phi(x)$, if $x \in \breve C$,
while if $x$ is such that $(y,b_k,x) \in E'$, where $y \in \breve C$,
then $\wt\phi(x) = \bar x$ is the unique element of $X'$ such that
$\bigl( \phi(y),b_k,\bar x) \in E'$. 
In terms of $K$-cosets, this means 
$$
\wt\phi(Khg_k) = K\bar hg_k\,, \quad \text{if} \quad y= Kh \in \breve C
\AND \phi(y) = K\bar h\,.
$$    
It is a straightforward exercise that $\wt \phi$ is well-defined and bijective. 
Now let $x, x' \in \wt C$ with $(x,a', x') \in E'$, where 
$a' \in \Si'$. We need to show that also $(\bar x, a', \bar x') \in E'$,
where $\bar x= \wt\phi(x)$ and $\bar x' =\wt\phi(x')$.

There are $y, y' \in \breve C$ and 
$k, l \in \{0, \dots, m \}$ such that $(y,b_k,x), (y',b_l,x') \in E'$
when $k \ge 1$, resp. $l \ge 1$, and $y=x$ when $k=0$, resp. $y'=x'$ when
$l=0$. In particular, 
$g_k\,\psi'(\bar a)\,g_l^{-1} \in H$, and there is a word $v \in  \Si^*$
such that $|v| \le 3M_2$ and $\psi(v) = g_k\,\psi'(\bar a)\,g_l^{-1}$. 
Let $\bar y = \phi(y) = \wt\phi(y)$
and $\bar y' = \phi(y') = \wt\phi(y')$. The path $\pi_y(v)$ from $y$ to
$y'$ lies entirely within $C$. Therefore it is mapped by $\phi$ to the
path $\pi_{\bar y}(v)$ from $\bar y$ to $\bar y'$, which lies entirely
within $C_i$. By the construction of $\wt\phi$, we now get that 
in $\wt C_i$ we have the edge $(x,\bar a, x')$.
Again, this can also be seen in terms of cosets: let $\bar y = K\bar h$.
Then we have shown that $\bar y' = K\bar h'$, where $\bar h' = \bar h\,\psi(v)$,
so that
$$ 
\bar x = K\bar h\,g_k \AND \bar x' = K\bar h'\,g_l = K \bar h\, \psi(v)\, g_l =   
K\bar h\,g_k\,\psi'(\bar a)\,.
$$
The rest of the proof evolves almost precisely as in Step 1 and is omitted.
\end{proof}

We obtain the following important class of graphs that have ``good''
tree sets.

\begin{thm}\label{thm:virtfree} Suppose that $X = X(G,K,\psi)$,  where
$G$ is virtually free, $\K$ is a finitely generated \emph{free}
subgroup of $G$, and  $\psi$ is is any symmetric semigroup presentation of $G$.
Then $X$ has a rooted, oriented tree set that is irreducible.
\end{thm}  

\begin{proof} The group $G$ has a free subgroup $\F$ with $[G:\F] < \infty$
that contains $\K$. 
This follows from the discussion in 
{\sc Scott}~\cite{Sc}, or also from
{\sc Bogopol'ski{\u\i}}~\cite[Thm. 9.1 \& proof]{Bog}.

Now, we have observed in  
\cite[Cor. 5.6]{CeWo3} that with respect to $\psi$ given by
the standard free generators of the free group $\F$ and their inverses, 
the Schreier graph $\XX(\F,\K)$ is such that outside a finite set $F$, its
cones are isomorphic with the cones of the standard Cayley graph of
the free group. In this way, we obtain a tree set which consists just
of those cones and inherits irreducibility from the tree associated
with the free group. 

We can now apply Proposition \ref{pro:inherit1} to conlcude the proof.
\end{proof} 


\section{Grammars associated with tree sets, and their dependency di-graphs}
\label{sec:grammars} 

The {\it dependency di-graph\/} $\Dcal = \Dcal(\Ccal)$ of a context-free grammar
$\Ccal = (\V, \Si, \Pb, S)$ is an oriented graph with vertex set $\V$, with an 
edge from $T$ to   $U$ (notation $T \to U$) if in $\Pb$ there is a production 
$T \vdash u$ with $u$ containing $U$. (Compare e.g. with {\sc Kuich}~\cite{Ku}.) 
We write $T \tos U$ if in $\Dcal$ there is an oriented  
path of length $\ge 0$ from $T$ to $U$. 

Consider the equivalence relation on $\V$ where $T \sim U$ if 
$T \tos U$ and $U \tos T$. The equivalence classes, denoted $\V(T)\,$ ($T \in \V$) 
are called the {\it strong components\/} 
of $\Dcal(\Ccal)$. The strong components are partially ordered: 
$\V(T) \pre \V(U)$ if there is an oriented path in $\Dcal$ from 
$T$ to $U$. 
A strong component is called \emph{essential}, if it is maximal in the above 
order. The graph is called \emph{strongly connected,} if the whole of
$\V$ is a single strong component. (These notions make sense for any oriented
graph.)  

The grammar $\Ccal$ (and the language it generates) is called \emph{ergodic},
if $\Dcal(\Ccal)$ is strongly connected. This notion was introduced 
by {\sc Ceccherini and Woess}~\cite{CeWo1}.

\begin{ass}\label{ass:umptions}
We consider an infinite symmetric, fully deterministic graph $\XX$ 
that admits a rooted tree set $\Ecal$ with finite type. By Proposition
\ref{pro:goodcuts} we may assume without loss of generality that $\Ecal$ is
separated and that
$\partial_{\Ecal} X$ and all $\partial_{\Ecal} C$ ($C \in \Ecal$) are connected.
We let $I = \{ 1, \dots, r\}$ be the set of types of cuts in $\Ecal$.
In addition, we set $C_0=\XX$ and $\partial C_0 = \partial_{\Ecal} C_0 = 
\partial_{\Ecal} \XX$ and consider all
maximal $C \in \Ecal$ as successors of $C_0$ (that is, $C_0 \to C$).
\end{ass}
 
 
In the proof of \cite[Thm. 4.2]{CeWo3}Theorem \ref{thm:characterize1}, 
we have constructed a deterministic pushdown automaton that accepts 
$L_{x_0,y_0}\,$, where $x_0, y_0 \in F$, and $F$ is the finite reference 
set with respect to which
the cones are defined (not necessarily a single root vertex). Replacing $F$ 
by $\partial_{\Ecal}\XX$ and the cones with 
respect  to $F$ by the cuts in $\Ecal$, we now construct an un-ambiguous grammar 
for the same language. We introduce the sets of variables 
\begin{equation}\label{eq:vars}
\begin{gathered}
\V_{\! 0} = \bigl\{ T_{x,y_0} :x \in \partial_{\Ecal} C_0\}\,, 
\quad
\V_{\! i} = \bigl\{ T_{x,y} :x \in \partial_{\Ecal} C_i\,, y \in  \partial C_i\}
\quad
(i \ge 1)\,, 
\AND\\
\V =\bigcup_{i=0}^r \V_{\! i}\,.
\end{gathered}
\end{equation}

For each $C \in \Ecal$, there are $i \in \{1, \dots, r\}$ and an isomorphism 
$\phi_C: C \to C_i$. For every $x \in X \setminus \partial_{\Ecal} X$, there
is a unique $C \in \Ecal$ such that $x \in \partial_{\Ecal} C$. We set 
$\phi(x)=\phi_C(x) \in \bigcup_{i=1}^r \partial_{\Ecal} C_i\,$.

The set of production rules $\Pb$ of our grammar is as follows. For all
$a, b \in \Si$, all $i,j \in \{0, \dots, r\}$ with $j \ge 1$, 
$x, x', x'' \in \partial_{\Ecal} C_i\,$,
$y \in  \partial C_i$
(with $y=y_0$ when $i=0$) and $\bar x, \bar y \in \partial C$,
where $C_i \to C$ in $\Ecal$ and $C$ is isomorphic with $C_j\,$,
\begin{equation}\label{eq:grammar}
\begin{array}{rll}
T_{y,y} &\vdash \epsilon 
\\
T_{x,y} &\vdash aT_{x',y}&\text{whenever}\; (x,a,x') \in E\,,\\
T_{x,y} &\vdash aT_{\phi(\bar x), \phi(\bar y)} bT_{x'',y}
 &\text{whenever}\; (x,a,\bar x)\,,\, (\bar y, b, x'') \in E\,.
\end{array}
\end{equation}
In passing, we observe that this
grammar has the so-called operator normal form (the right hand sides
of the productions contain no subword in $\V^2$).

\begin{lem}\label{lem:gram} \emph{(a)} Let the start symbol be 
$T_{x,y} \in \V_{\! i}\,$, where $i \in \{ 0, \dots, r\}$.
Then the grammar 
$\Ccal = (\V, \Si, \Pb, T_{x,y})$, with $\V$ as in \eqref{eq:vars} and $\Pb$ as in
\eqref{eq:grammar}, is un-ambiguous and generates the language
$$
L_{x,y}(C_i) = \{ \ell(\pi) : \pi \in \Pi_{x,y} \;\text{and}\;\pi \subset C_i\}\,.
$$
\emph{(b)} 
The dependency digraph $\Dcal(\Ccal)$ has the following properties. 
\begin{enumerate}
\item If $x,x'\in \partial_{\Ecal} C_i, y \in \partial C_i$ (with $y=y_0$ when $i=0$)
then there is an oriented 
path in $\Dcal(\Ccal)$ from $T_{x,y}$ to $T_{x',y}\,$.
\item
If $C_i \to C \in \Ecal$, where $C$ has type $j$, and if  
$x \in \partial_{\Ecal} C_i\,$, $y \in \partial C_i\,$,  
$x' \in \partial_{\Ecal} C_j$ and $y' \in \partial C_j\,$, then
there is an oriented path in $\Dcal(\Ccal)$ from $T_{x,y}$ to $T_{x',y'}\,$.
\end{enumerate}
\emph{(c)} In particular, if $\Ecal$ is also irreducible, then 
the strong compontents of $\Dcal(\Ccal)$ are $\V_{\! 0}$ and 
$\V_{\! \text{\rm ess}} = \V_1 \cup \dots \cup \V_s\,$, and 
$\V_{\! 0} \pre \V_{\! \text{\rm ess}}\,$.
\end{lem}

\begin{proof} (a) 
The language generated by $T_{x,y}$ is $L_{x,y}(C_i)$, because each 
non-trivial path from  $x$ to $y$ within $C_i$ 
(where $x \in \partial_{\Ecal} C_i$ and $y \in \partial C_i$) decomposes 
uniquely in the following way. 

- Either its first edge goes from $x \in \partial_{\Ecal} C_i$ to some 
$x' \in \partial_{\Ecal} C_i\,$,
and it is followed by a path from $x'$ to $y$ within $C_i\,$, 
 
- or else the first edge goes to a vertex $\bar x \in \partial C$, where 
$C_i\to C$ (and $C$ is isomorphic with some $C_j$),  
this is followed by a path within $C$ from $\bar x$ to some 
$\bar y \in \partial C$ and an edge from $\bar x$ to some $x'' \in
\partial_{\Ecal} C_i\,$, and then it terminates with a path from $x''$ to $y$ 
within $C_i\,$.
The reader is invited to draw a figure.

The grammar is un-ambiguous by its construction.
\\[5pt]
(b) Regarding (1), since $\partial_{\Ecal}C_i$ is connected, there is a path 
in the graph $\XX$  from $x$ to $x'$ that lies within $\partial_{\Ecal}C_i$. 
Using the second type of the production rules in \eqref{eq:grammar},
we see that this translates into a path in the graph $\Dcal(\Ecal)$ 
from $T_{x,y}$ to $T_{x',y}\,$. It lies within $\V_{\! i}\,$.

(2) Let $C$ be a successor cone of $C_i$ that has type $j$, and let 
$x' \in \partial_{\Ecal} C_j$ and $y' \in \partial C_j$ be as stated. 
By (1), it is sufficient to consider only the case when $x' \in \partial C_j$.
Let  $\bar x, \bar y \in \partial C$ be the (unique) elements for which 
$\phi(\bar x)=x'$ and $\phi(\bar y)=y'$. There must be 
$x'', x''' \in \partial_{\Ecal} C_i$ and $a,b \in \Si$ such that $(x'',a,\bar x),
(\bar y, b, x''') \in E$. 

By (1), in $\Dcal(\Ccal)$ there is a path from $T_{x,y}$ to $T_{x'',y}\,$.
In our grammar, there is the production 
$T_{x'',y} \vdash aT_{x',y'} bT_{x''',y}\,$. Thus, in $\Dcal(\Ccal)$ there is 
an edge from $T_{x'',y}$ to $T_{x',y'}\,$, and we also get a path
from $T_{x,y}$ to $T_{x',y'}\,$. This proves (2) and completes (b).
\\[5pt]
(c) is an immediate consequence of (b).
\end{proof}

\begin{cor}\label{cor:cuts-cf}
Let $\Ecal$ be a rooted, oriented tree set consisting of
cuts of the fully deterministic, symmetric labelled  graph $X$.
Suppose that $\Ecal$ tesselates $\XX$ and has finite type.

Then $L_{x,y}$ is context-free for every $x, y \in X$.

In particular, $\XX$ is a context-free graph in the sense of 
\cite{MS2}.
\end{cor}

\begin{rmk}\label{rmk:non-symm}
When we apply the above to a Schreier graph $\XX(G,K,\psi)$, the assumption of 
symmetry of $\psi$ is not crucial. Suppose the graph has a tree set 
with the above properties for some (and hence every) symmetric semigroup
presentation. If $\psi$ is not symmetric then we can symmetrise it by adding
further elements to $\Si$ and defining an appropriate symmetric 
extension $\ol\psi$ of $\psi$ such that each oriented edge of $X(G,K,\psi)$
is also an edge of $X(G,K,\ol\psi)$, and for every pair of inverse edges of 
$X(G,K,\ol\psi)$, at least one is an edge of $X(G,K,\psi)$. 
If we have a tree set for the symmetrised graph, then we can also use it
for the original graph. That is, while our cuts come from $X(G,K,\ol\psi)$,
in the resulting grammar \eqref{eq:grammar}, we refer to the edge set
$E$ of $X(G,K,\psi)$. This grammar
will generate the required languages $L_{x,y}$ for the original 
non-symmetric graph.

One problem, however, is that when the tree set of $\XX(G,K,\ol\psi)$
is irreducible, it is not necessarily true that the grammar
associated with  $\XX(G,K,\psi)$ has the properties derived in Lemma
\ref{lem:gram}(c). This is the primary obstacle when one wants
to extend the results of the next sections directly to the non-symmetric case.
\end{rmk}

We can adopt the more general definition of a fully deterministic, symmetric 
graph $\XX$ to be context-free when
it has a tree-set  of finite type that tesselates $\XX$. We see that this 
is in reality equivalent with the definition in terms of cones with respect
to a finite set $F$, or more specifically, with respect to a single point $o$.

The reason why we have embarked on the fatigue regarding cuts and tree
sets is that we want to have a ``good'' grammar, that is, one as  
in Lemma \ref{lem:gram}(c). This will have an interesting application,
as we shall see in the next section. The collection of all cones with respect 
to a given set $F$ is of course a 
special case of a tree set. Our main problem was that when working
just with cones, it is by now means clear how to show that irreducibility
can in some suitable way be transferred to the cones of a symmetric 
Schreier graph $ X(G,K,\psi')$ from those of $ X(H,K,\psi)$, where 
$[G:H] < \infty$. For tree sets, we can do this, so that we get a ``good''
grammar, for example, for all symmetric Schreier graphs 
$X(G, \K, \psi)$, where $G$ is a virtually free group and $\K$ is a finitely
generated free subgroup of $G$. This applies, in particular, to
Cayley graphs of virtually free groups.
 
\section{Generating functions, tree sets, and random walks}\label{sec:genfun}

Let $\Ccal = (\V,\Si,\Pb,S)$ be a \emph{non-ambiguous} context-free grammar.
With each element $a \in \Si$, we associate a positive real number $\mu(a)$.
Furthermore, let $z$ be a complex variable.
We also associate a complex variable $y_T$ with every $T \in \V$ and
let $\yy = (y_T)_{T \in \V}\,$, a row vector.
Then we define 
\begin{equation}\label{eq:substitute}
\begin{gathered}
\zeta(a) = \mu(a) z  \;\; \text{for}\;\; a \in \Si\,,\quad
\zeta(T) = y_T  \;\; \text{for}\;\; T \in \V\,,\AND\\
\zeta(u) = \zeta(u_1) \cdots \zeta(u_n) \,,\quad \text{if}\; 
u=u_1 \cdots u_n  \in (\Si \cup \V)^*\,.
\end{gathered}
\end{equation}
With the language $L_T$ generated by $T \in \V$, we associate the 
generating function
\begin{equation}\label{eq:f_T}
f_T(z) = \sum_{w \in L_T} \zeta(w)\,.
\end{equation}
This defines an analytic function in a neighbourhood of the origin in
$\C^{|\Si|}$. Indeed, if $\mu(a)|z| < 1/|\Si|$ for every $a \in \Si$ then the 
series converges absolutely.
With $T$, we also associate the polynomial
\begin{equation}\label{eq:P_T}
\Pol_T(z;\yy) = \sum_{u\,:\, T \vdash u} \zeta(u)
\end{equation}
in $\yy$ and $z$. 

We can assume without loss of generality (up to a simple modification
of the production rules that preserves un-ambiguity) that our
grammar has no 
chain rules (productions of the form $T \vdash U$, where $T, U \in \V$).

Then the fundamental theory of {\sc Chomsky and
Schutzenberger}~\cite{ChSc} implies that the functions $f_T(z)$
satisfy the system of equations
\begin{equation}\label{eq:eqs}
f_T(z) = \Pol_T\bigl(z; f_U(z), U \in \V\bigr)\,,\; T \in \V\,,
\end{equation}
where each variable $y_U$ has been replaced by the function $f_U(z)$.
Compare also with {\sc Kuich and Salomaa}~\cite[\S 14]{KuSa}.

\smallskip

We return to the assumptions \ref{ass:umptions} regarding the graph 
$\XX$ and the tree set $\Ecal$, which is also assumed to be irreducible. 
From now on, the grammar $\Ccal$ will always be the one of 
\eqref{eq:grammar}. We let $\Si$ be the (symmetric) label alphabet of $\XX$.
In \eqref{eq:substitute}, we assume that $\mu(a) > 0$ for all
$a \in \Si$. It is no loss of generality to require that 
$\mu(\Si) = \sum_{a \in \Si} \mu(a)= 1$; otherwise, we just can
replace $z$ with $z/\mu(\Si)$. Then $\mu$ gives rise to a \emph{random walk} 
on $X$. This is the time-homogeneous Markov chain $(Z_n)_{n\ge 0}$ whose state 
space is $\XX$ and whose one-step transition probabilities are 
\begin{equation}\label{eq:transprob}
p(x,y) = \sum_{a \in \Si\,:\,(x,a,y) \in E} \mu(a)\,.
\end{equation}
See the monograph by {\sc Woess}~\cite{Wbook} for an outline of the theory
of random walks on infinite graphs and groups. We use the terminology
of that reference.

We write $p^{(n)}(x,y) = \Prob[Z_n=y\mid Z_0=x]$ for the probability that 
the random walk starting at $x$ is in $y$ after $n$ steps. This is just
the $(x,y)$-entry of the matrix power $P^n$, where 
$P=\bigl(p(x,y)\bigr)_{x,y\in X}$ is the transition matrix.
The \emph{Green function} of the random walk is
\begin{equation}\label{eq:green}
G(x,y|z) = \sum_{n=0}^{\infty} p^{(n)}(x,y)\,z^n = G(x,y|z)\,, \quad z \in \C\,.
\end{equation}
It is a well known consequence of connectedness of $\XX$ (i.e., irreducibility
of the transition matrix) that its radius of convergence 
\begin{equation}\label{eq:Rmu}
R_{\mu} = 1 \big/ \limsup_n p^{(n)}(x,y)^{1/n}
\end{equation}
is the same for all $x,y \in X$ and that either $G(x,y|R_{\mu}) < \infty$
for all $x,y$ or $G(x,y|R_{\mu}) = \infty$ for all $x,y$.

For the variables of the grammar $\Ccal$, the associated generating functions
now have the following interpretations:
\begin{enumerate}
\item If $T = T_{x,y_0}\in \V_{\! 0}$, where $x, y_0\in \partial_{\Ecal} X$ 
then
$$
f_T(z) = G(x,y_0|z)\,.
$$
\item If $T = T_{x,y}\in  \V_{\!\text{\rm ess}}\,$, where 
$x \in \partial_{\Ecal}C_i\,, y \in \partial C_i$ with $i \ge 1$ then 
$$
f_T(z) = \sum_{n=0}^{\infty} p_{C_i}^{(n)}(x,y)\, z^n\,,
$$
where $p_{C_i}^{(n)}(x,y)$ is the probability that the random walk starting at
$x$ is in $y$ at the $n$-th step, without having left $C_i\,$.  
\end{enumerate}

For $T \in \V$, we write $\delta_T=1$
if there is a production $T \vdash \epsilon$, and $\delta_T=0$, otherwise.
Then we have for our grammar
\begin{equation}\label{eq:poly}
\Pol_T(z;\yy) = \delta_T + \sum_{T \vdash aU} \mu(a)\,z\, y_U
+ \sum_{T \vdash aVbU} \mu(a)\,\mu(b)\,z^2\, y_V\,y_U\,,
\end{equation} 
where the sums range over all productions of \eqref{eq:grammar}
whose left hand side is $T \in \V$.

We study analytic properties of the associated generating functions
according to \eqref{eq:substitute} and \eqref{eq:f_T}. 

We first restrict the grammar to $\V_{\!\text{\rm ess}}\,$, as defined in
Lemma \ref{lem:gram}. That is, we
keep only those productions whose left hand sides belong to
$\V_{\!\text{\rm ess}}\,$. Note that all variables occuring in its right hand 
sides also belong to $\V_{\!\text{\rm ess}}\,$.
We get a grammar $\Ccal_{\text{\rm ess}}$ that is un-ambiguous and ergodic,
and there are productions whose right hand sides contain two variables
(the grammar is \emph{non-linear}).

Consider the associated system of equations \eqref{eq:eqs}. 
Equations of this type occur quite frequently. 
We appeal to methods that have been developped in the context of random
walks on free groups and trees by {\sc Lalley}~\cite{La1}, \cite{La2}; 
compare also with the exposition in {\sc Woess}~\cite[\S19.B]{Wbook} and with
{\sc Nagnibeda and Woess}~\cite{NaWo}. In the more general 
context of analytic combinatorics, see the monographs of {\sc Flajolet and
Sedgewick}~\cite[\S VII.6]{FlSe} and {\sc Drmota}~\cite[\S 2.2.5]{Dr}. 
For the setting of generating functions associated with grammars, see also
\cite{CeWo1}. 
Appealing to \cite[Thm. VII.6]{FlSe}, we get the following.

\begin{pro}\label{pro:genfunprop} \emph{(a)} All the power series $f_T(z)$, 
$T = T_{x,y} \in \V_{\!\text{\rm ess}}\,$, have the same radius of convergence
$R$.
\\[4pt]
\emph{(b)} One has $f_T(R) < \infty\,$, and $z=R$ is a
simple (i.e., quadratic) branching point of each $f_T(z)$. There
are functions $g_T(z)$ and $h_T(z)$, which near $z=R$ are analytic and 
real-valued for real $z$, such that $h_T(R) > 0$ and the identity
$$
f_T(z) = g_T(z) - h_T(z)\sqrt{R-z} 
$$
%
is valid in a neighbourhood of the point $R$ in $\C$, except for real $z > R$.
\end{pro}

Now that we have good information about the generating functions associated
with the variables in $\V_{\!\text{\rm ess}}\,$, let us investigate how
they determine the ``remaining'' functions $f_T(z)$, $T \in \V_{\! 0}\,$.
For such $T$, the polynomial $\Pol_T(z;\yy)$ of \eqref{eq:poly} is such
that of the variables $U,V$ appearing there, one has $U \in \V_{\! 0}$ and
$V \in \V_{\!\text{\rm ess}}\,$. That is, the column vector 
$$
\ff(z)= \bigl( f_T(z) \bigr)_{T \in \V_{\! 0}}
$$ 
is determined by the linear system
\begin{equation}\label{eq:linsyst}
\ff(z) = \ee + Q(z)\ff(z)\,,
\end{equation}
where $\ee = \bigl(\delta_T\bigr)_{T \in \V_{\! 0}}\,$ (another column vector),
and $Q(z) = \bigl(q_{T,U}(z)\bigr)_{T,U\in\V_{\! 0}}$ is the matrix 
over $\V_{\! 0}$ 
with entries
\begin{equation}\label{eq:qTU} 
q_{T,U}(z) = \sum_{T \vdash aU} \mu(a)\,z 
+ \sum_{T \vdash aVbU} \mu(a)\,\mu(b)\,z^2\, f_V(z)\,,
\end{equation}
where the first sum ranges over all $a \in \Si$ such that we have
the production $T \vdash aU$, and the second sum ranges over
all $a, b \in \Si$ and $V \in \V_{\!\text{\rm ess}}$ such that we have
the production $T \vdash aUbV$. Note that there is at least one
term of this last type, that is, $Q(z)$ contains at least one of the
functions $f_V(z)$, $V \in \V_{\!\text{\rm ess}}\,$, in one of its terms. 

\begin{lem}\label{lem:irred}
For $0 < z \le R$, the non-negative matrix $Q(z)$ is irreducible and depends
continuously on $z$ (analytically for $0 < z < R$).
Furthermore, in a complex neighbourhood of the point $R$, except for real 
$z> R$, we can decompose 
$$
Q(z) = A(z) - \sqrt{R-z}\,B(z)\, 
$$
where the matrices $A(z)$ and $B(z)$ are analytic functions of $z$
near $R$, and $B(R)$ is a non-negative matrix that does not vanish. 
\end{lem}

\begin{proof} It is clear that our matrix is continuous for $0 \le z \le R$,
resp. analytic for $0 \le z < R$.
Now recall that each variable in $\V_{\! 0}$ has the form $T=T_{x,y_0}\,$,
where $x \in \partial_{\Ecal} X$. We may as well replace each $T$ with the corresponding
$x$ and index $\ff(z)$, $\ee(z)$ and $Q(z)$ accordingly. In particular, when $T=T_{x,y_0}$
and $U=T_{x',y_0}$ then $q_{T,U}(z) = q_{x,x'}(z)$, and we see
that in this notation, 
\begin{equation}\label{eq:barp}
q_{x,x'}(z) \ge \ol p_{x,x'}z\,,\quad\text{where}\quad
\ol p_{x,x'} = \sum_{a\,:\, (x,a,x') \in E} \mu(a)\,,
\quad\text{if} \; x,x' \in \partial X\,.
\end{equation}
But already the matrix 
$\ol P = \bigl(\ol p_{x,x'} \bigr)_{x,x'\in \partial_{\Ecal} X}$ is 
irreducible, since by our construction, $\partial_{\Ecal} X$ is a connected 
subgraph of $X$. 

The decomposition of $Q(z)$ is an immediate consequence of 
\eqref{eq:qTU} and Proposition \ref{pro:genfunprop}.
\end{proof}

We shall need the following ``side-product'' of the Perron-Frobenius theory of
non-negative matrices. 

\begin{lem}\label{lem:laprime}
Let $M(z)$, $z \in (\alpha\,,\,\beta) \subset \R$, be irreducible, 
non-negative square matrices of the same dimension whose entries are 
real-analytic functions of $z$. 
Let $\la(z)$ be the largest positive eigenvalue of $M(z)$, and let
$\vv(z)^t$ and $\ww(z)$ be the (strictly positive) left and right 
Perron-Frobenius eigenvectors, that is
$$
\vv(z)^t M(z) = \la(z)\cdot \vv(z)^t \AND M(z)\ww(z) = \la(z)\cdot \ww(z)\,,
$$
normalized such that 
$\langle \vv(z),\uno\rangle = \langle \vv(z),\ww(z)\rangle = 1\,.$
(Here, $\langle \cdot,\cdot \rangle$ is the inner product, and $\uno$
is the vector with all entries $=1$.)

Then $\la(z)$, $\vv(z)$ and $\ww(z)$ are analytic functions of 
$z \in (\alpha\,,\,\beta)$, and 
$$
\la'(z) = \vv(z)^t M'(z) \ww(z)\,.
$$
(Derivatives with respect to $z$.)
\end{lem} 

\begin{proof}
The fact that $\la(z)$, $\vv(z)$ and $\ww(z)$ are analytic functions is
straightforward; compare e.g. with \cite[Lemma 7.6]{La2}.
Let $z, z_0 \in (\alpha\,,\,\beta)$.
Then
$$
\vv(z_0)^t\bigl( M(z) - M(z_0) \bigr) \ww(z) = \la(z) - \la(z_0)\,.
$$
Dividing by $z-z_0$ and letting $z \to z_0\,$, we get the formula for
$\la'(z)$.
\end{proof}

\begin{pro}\label{pro:G-expand} For the Green function associated with
the random walk, there is the following alternative between cases (a), (b) or
(c), where $R$ is as in Proposition \ref{pro:genfunprop}(a).

\smallskip

\noindent
\emph{(a)} $\;R_{\mu} < R$, and for every pair $x,y \in \XX$, the singularity
$z=R_{\mu}$ is a simple pole of $G(x,y|z)$.

\smallskip

\noindent
\emph{(b)} $\;R_{\mu} = R$, and for every pair $x,y \in \XX$,
there are functions $g_{x,y}(z)$ and 
$h_{x,y}(z)$, which near $z=R$ are analytic and 
real-valued for real $z$, such that $h_{x,y}(R) > 0$ and the identity
$$
G(x,y|z)= g_{x,y}(z) + h_{x,y}(z)\big/\sqrt{R_{\mu}-z} 
$$
is valid in a complex neighbourhood of $R_{\mu}$ in $\C$, except for real 
$z > R_{\mu}\,$.
 
\smallskip

\noindent
\emph{(c)} $\;R_{\mu} = R$, and for every pair $x,y \in \XX$,
there are functions $g_{x,y}(z)$ and 
$h_{x,y}(z)$, which near $z=R$ are analytic and 
real-valued for real $z$, such that $h_{x,y}(R) > 0$ and the identity
$$
G(x,y|z) = g_{x,y}(z)- h_{x,y}(z)\,\sqrt{R_{\mu}-z}
$$
is valid in a complex neighbourhood of $R_{\mu}$ in $\C$, except for real 
$z > R_{\mu}\,$.
\end{pro}
 
\begin{proof} The elements of the vector $\ff(z)$ in equation
\eqref{eq:linsyst} are the functions $G(x,y_0|z)$,
where $x \in \partial_{\Ecal}X$. We shall apply the last lemma to
$Q(z)$.
For $0 \le z \le R$, let $\la(z)$ be its largest real 
eigenvalue. If $z=0$ then $Q(z)$ is the zero matrix, so that $\la(0)=0$.
Otherwise, by Lemma \ref{lem:irred}, the Perron-Frobenius theorem applies,
and $\la(z)$ is a simple zero of the characteristic polynomial 
$\chi(\la,z) = \det \bigl(\la\cdot I - Q(z)\bigr)$, where $I$ is
the identity matrix over~$\V_{\! 0}\,$.

For $z \in (0\,,\,R)$, we can write 
$$
\chi(1,z) = \chi(1,z) - \chi \bigl(\la(z),z\bigr) 
= \bigl(1-\la(z)\bigr)\,\eta(z)\,,
$$ 
where $\eta(z)$ is analytic and non-zero.

Now let $A(z)$ be the ``adjunct'' matrix of $I-Q(z)$, whose $(T,U)$-entry
is
$$
(-1)^{\pm 1} \det\bigl(I-Q(z) | U,T\bigr)\,,
$$ 
with the row of $U$ and column of $T$ deleted from the matrix $I-Q(z)$, and
the sign according to the parity of the position of $(U,T)$.
Then $A(z)$ is real-analytic for $z \in (0\,,\,R)$, and
$\bigl(I-Q(z)\bigr)A(z) = \chi(1,z) \cdot I$.

\smallskip

(a) If $\la(R) > 1$ then there is a unique $z_0 \in (0\,,\,R)$
such that $\la(z_0)=1$. For $z < z_0\,$, the matrix $I - Q(z)$ is invertible,
whence $\ff(z)$ is analytic.

Now note that $Q'(z) \ge \ol P$, where $\ol P$ is the irreducible
matrix defined in the proof of Lemma \ref{lem:irred}. It is the restriction
of the transition matrix of our random walk to $\partial_{\Ecal}X\,$.

In particular, 
$\la'(z_0) > 0$ by Lemma \ref{lem:laprime}. Therefore 
$1- \la(z) = (z_0-z)\gamma(z)\,$,
where $\gamma(z)$ is analytic near $z=z_0$ and $\gamma(z_0) > 0$.
We can write for $z$ near $z_0$
$$
\begin{aligned}
\bigl(I-Q(z)\bigr)^{-1} &=\frac{1}{\chi(1,z)}A(z) 
= \frac{1}{1-\la(z)}\frac{1}{\eta(z)} A(z) \\
&= \frac{1}{z_0-z}B(z)\,, \quad \text {where}\\
B(z) &= \frac{1}{\gamma(z)\eta(z)} A(z)  
\end{aligned}
$$
If $z$ is close to $z_0$ then $A(z)$ is close to $A(z_0)\,$. This last
matrix is strictly positive in each entry, since $A(z_0)= 
\frac{\partial}{\partial\la}\chi(1,z_0) \cdot \ww(z_0)\vv(z_0)^t$, where $\vv(z_0)_t$
and $\ww(z_0)$ are the left and right Perron-Frobenius eigenvectors of
$Q(z_0)$, see \cite{Se}. Therefore, for real $z$ close to
$z_0$, the matrix $B(z)$ has all entries strictly positive.
We obtain near $z_0$
$$
\ff(z) = \frac{1}{z_0-z}B(z)\ee\,. 
$$
Since $R_{\mu}$ is the smallest positive singularity of each of the functions
$G(x,y|z)$, we see that $R_{\mu}=z_0$, which is a simple pole of $G(x,y_0|z)$,
where $x \in  \partial_{\Ecal} X$.

\smallskip

We shall argue at the end of this proof that the same must be true for all
$x,y \in X$; compare once more with \cite{La2}.

\smallskip

(b) Suppose next that $\la(R)=1$. Then $I-Q(z)$ is invertible for all
$z \in [0\,,\,R)$, and $\ff(z)$ is analytic for all those $z$. Since
$R_{\mu} \le R$ always, we have $R_{\mu}=R\,$.

Now we substitute $\zf = \sqrt{R-z}$ (for complex $z$, unless $z > R$
is real). Then the expansion of Proposition \ref{pro:genfunprop}(b),
for $V \in \V_{\!\text{\rm ess}}$ instead of $T$, can be written  as 
$$
\begin{gathered}
f_V(z) = \wt f_V\bigl(\sqrt{R-z}\,\bigr)\,,\quad\text{where}\quad
\wt f_V(\zf) = \wt g_V(\zf) - \zf \,\wt h_V(R-\zf^2) \quad  \text{with}\\
\wt g_V(\zf) = g_V(R-\zf^2) \AND \wt h_V(\zf) = h_V(R-\zf^2).
\end{gathered}
$$
$\wt f_V(\zf)$ is an analytic function near the origin, and 
$\wt f_V'(0) = -h_T(R) < 0$. When we carry over this subsitution to the matrix
$Q(z)$, then we find near $z=R$ that 
$Q(z) = \wt Q\bigl(\sqrt{R-z}\,\bigr)\,$, where the matrix $\wt Q(\zf)$
is analytic near $\zf=0$, non-negative irreducible for $\zf > 0$, and 
$$
\wt Q'(0) = \bigl( \wt q{\,}'_{T,U}(0) \bigr)_{T, U \in \V_{\! 0}}
\quad \text{with}\quad
\wt q{\,}'_{T,U}(0) = - \sum_{T \vdash aVbU} \mu(a)\mu(b)\, R^2 \, h_V(R)\,.
$$
As above, the last sum is over all productions $T \vdash aVbU$
with $a, b \in \Si$, $V \in \V_{\!\text{\rm ess}}\,$. Since
there is at least one such production for some $T, U \in \V_{\! 0}\,$,
we see that $\wt Q'(0)$ is non-positive and strictly negative in some entry.
Lemma \ref{lem:laprime} implies that $\wt\la'(0) < 0$, where $\wt \la(\zf)$
is the largest positive eigenvalue of the non-negative matrix $\wt Q(\zf)$
for $\zf > 0$.

Arguing precisely as in case (a), we now deduce that for $\zf$ near $0$,
$$
\bigl(I-\wt Q(\zf)\bigr)^{-1} = \frac{1}{\zf}\wt B(\zf)\,, 
$$
where $\wt B(\zf)$ is a matrix with analytic entries that are strictly 
positive when $\zf \ge 0$. Re-substituting $\zf = \sqrt{R-z}$, we obtain
the desired expansion for $G(x,y_0|z)$ for $x \in \partial_{\Ecal}X$.

\smallskip

(c) If $\la(R)<1$ then $R_{\mu}=R$ by the same reason as in case (b).
We make the same substitution as in case (b) but observe that
$I - Q(z)$ is invertible for all $z \in [0\,,\,R]$, so that
$I - \wt Q(\zf)$ is invertible for all $\zf$ close to $0$, and 
$$
\bigl(I- \wt Q(\zf)\bigr)^{-1} = \wt B(\zf) 
= \sum_{n=0}^{\infty} \wt Q(\zf)^n\,. 
$$
This is a strictly positive matrix with analytic entries, and once more 
$\ff(z) = \wt B\bigl( \sqrt{R-z} \bigr) \ee$. In order to make sure
that in the proposed singular exapnsion, we really have $h_{x,y_0}(R) > 0$,
we must check that the matrix $\wt B'(0)$ is strictly negative in
each entry:
\begin{equation}\label{Btildepr}
\wt B'(0) = \left(\sum_{n=1}^{\infty} n\wt Q(0)^{n-1}\right) \wt Q'(0)\,.
\end{equation} 
The matrix sum in the parentheses is absolutely convergent and, by
irreducibility of $\wt Q(0)$, strictly positive in each entry.
We have seen above that $\wt Q'(0)$ is non-positive and strictly negative 
in some entry. 

\smallskip

We conclude that the alternative between cases (a), (b) and (c) 
holds for all Green functions $G(x,y_0|z)$, where $x \in \partial_{\Ecal}\XX$.
Given arbitrary $x,y \in \XX$, we know that we can modify $\Ecal$ so that
$x,y \in \partial_{\Ecal}\XX$, and we can set $y_0=x$.

Thus, for every individual choice of $x, y \in X$, one of the three
singular expansions of (a), (b) or (c) is valid. For the following argument, 
compare once more with \cite{La2}.
 
Let $x,y,x',y' \in \XX$. Then, by connectedness of $X$, there are $k, l \in \N$
such that $p^{(k)}(x,x') > 0$ and $p^{(l)}(y',y) > 0$, and 
$p^{(k+n+l)}(x,y) \ge p^{(k)}(x,x')p^{(n)}(x',y')p^{(l)}(y',y)$ for all $n$. 
Therefore, setting $C = C(x,y,x',y') = p^{(k)}(x,x')p^{(l)}(y',y)/2^{k+l}$,
we have $C>0$ and
$$
G(x,y|z) \ge C \,G(x',y'|z) \quad\text{for }\; 1/2 \le z < R\,.
$$ 
An analogous inequality holds when $x,y$ is exchanged with $x',y'$.
That is, 
$$
0 < \lim_{z\to R} G(x,y|z)/G(x',y'|z) < \infty\,.
$$
Thus, $G(x,y|z)$ and $G(x',y'|z)$ cannot have different expansions among the
three possible types. 
\end{proof}

Set $\dsf=2$ if $\XX$ is a bipartite graph (i.e., it has no odd cycles $\equiv$
closed paths with odd length), and $\dsf=1$, otherwise. 
This is the \emph{period} 
of the random walk. Then for every $x,y \in \XX$, we have that 
$p^{(n)}(x,y) > 0$ for all but finitely many $n$ for which $\dsf$ divides 
$n-d(x,y)$. The \emph{strong period} is
$$
\dsf_s = \gcd \{ n \in \N : \inf_{x\in \XX} p^{(n)}(x,x) > 0\}\,,
$$
if the latter set is non-empty. In our case, we always have
$p^{(2)}(x,x) \ge \sum_{a \in \Si} \mu(a)\mu(a^{-1})$, so that
$\dsf_s \in \{1,2\}$.

\begin{lem}\label{lem:strongper} We have $\dsf_s=1$ if and only if there is
a cut $C_i \in \Ecal$ ($i \ge 1$) that contains an odd cycle.
\end{lem}

\begin{proof} If $\dsf_s=1$ then there is an odd $m$ such that every vertex
lies on a cycle with length $m$. If $x$ lies in $C_i$ and 
$d(x,\partial C_i) \ge m/2$ then that cycle must lie within $C_i\,$.

Conversely, assume that some $C_i \in \Ecal$ ($i \ge 1$) contains an
odd cycle. Let $r$ be its length. There is $m_1$ such that every 
$y \in \partial C_i$ is at distance at most $m_1$ from that cycle.

Since $\Ecal$ is irreducible and
tesselates $\XX$, there is $m_2$ such that for every $x \in \XX$ there is
a cone $C$ of type $i$ such that $d(x,\partial C) \le m_2$.
We can construct a path 
from $x$ to a copy of our cycle within $C$ which has length  
$s \le m_1+m_2$. Thus, $x$ lies on a cycle (possibly with repeated edges)
with length $r + 2s$. By going back and forth along some edge, this
can be extended to a cycle with length $m= r+2(m_1+m_2)$, wich is odd.
But then $p^{(m)}(x,x) \ge \bigl( \min_{a \in \Si} \mu(a)\bigr)^m$
for every $x$. Therefore $\dsf_s=1$.  
\end{proof} 

Thus, we have three possible cases: $\dsf = \dsf_s = 1$ (in the situation
of Lemma \ref{lem:strongper}), $\dsf = \dsf_s = 2$ (when $X$ is a bipartite
graph), or $\dsf =1$ and $\dsf_s=2$. 

A result of {\sc Cartwright}~\cite{Ca} applies, see 
\cite[Thm. 9.4 \& page 96]{Wbook}. It yields the following.

\begin{lem}\label{lem:goodperiod}
The only singularities of  each of the functions $G(x,y|z)$ on 
the circle of convergence  $\{ z : |z|=R_{\mu}\}$ of these power series 
are $R_{\mu}$ and possibly $-R_{\mu}\,$. In addition,

\smallskip 

\noindent\emph{(a)}
if $\dsf = \dsf_s = 1$ then $R_{\mu}$ is the only singularity, and

\smallskip 

\noindent\emph{(b)} If $\dsf = \dsf_s = 2$ then both $R_{\mu}$ and $-R_{\mu}$
are singularities.
\end{lem}

In the second of those cases, $G(x,y|\!-\!z) = (-1)^{d(x,y)}\, G(x,y|z)$, 
so that the singular expansion of $G(x,y|z)$ at $z=-R_{\mu}$ is immediate
from the one at $z=R_{\mu}\,$, as described in Proposition \ref{pro:G-expand}.
In the remaining case, this is more complicated.

\begin{pro}\label{pro:-G-expand} Suppose that $\dsf =1$ and $\dsf_s=2$ for our 
random walk. Then for all $x,y \in X$, the Green function has the following 
expansion at $z=-R_{\mu}\,$.

\smallskip

\noindent
\emph{(1)} If $\;R_{\mu} < R$, then $G(x,y|z)$ is analytic at $z=-R_{\mu}\,$.

\smallskip

\noindent
\emph{(2)} If $\;R_{\mu} = R$, then there are functions $\bar g_{x,y}(z)$ and 
$\bar h_{x,y}(z)$, which near $z=-R_{\mu}$ are analytic and real-valued for 
real $z$, such that the 
identity
$$
G(x,y|z)= \bar g_{x,y}(z) - \bar h_{x,y}(z) \,\sqrt{R_{\mu}+z} 
$$
is valid in a complex neighbourhood of $-R_{\mu}$ in $\C$, except for real 
$z < -R_{\mu}\,$.

Furthermore, in case (c) of Proposition \ref{pro:G-expand}, one has
$|\bar h_{x,y}(-R_{\mu})| < h_{x,y}(R_{\mu})$.
\end{pro}
 
\begin{proof} We continue to use the notation of the proof of Proposition
\ref{pro:G-expand}. 

Let $x,y \in X$. We know from Lemma \ref{lem:strongper} that each cut $C_i\,$,
$i \ge 1$, is bipartite, while there is some odd cycle in $X$. By  
Proposition \ref{pro:goodcuts}(3), we can assume without loss of generality
that $\Ecal$ is such that this odd cycle as well as $x,y$ are contained in 
$\partial_{\Ecal}X$. We choose $y_0=y$.

\smallskip 

\noindent
\emph{Claim.} For every $z \in \C$ with $|z|\le R$ and $z \ne |z|$,
each eigenvalue $\la$ of $Q(z)$ satisfies $|\la| < \la(|z|)$.

\smallskip in case (c) of Proposition \ref{pro:G-expand},

\noindent
\emph{Proof.} Let $z$ be as stated.
From \eqref{eq:qTU}, we see that $|q_{T,U}(z)| \le q_{T,U}(|z|)$
for all $T, U \in \V_{\! 0}\,$. It follows that $|\la| \le \la(|z|)$
for every eigenvalue $\la$ of $Q(z)$. 

So let $\la \in \C$ with $|\la|= \la(|z|)$. The claim will be proved
when we show that $\la \, I - Q(z)$ is invertible. 

Consider the matrix $\ol P$ of \eqref{eq:barp}.
Since $\partial_{\Ecal} X$ is connected and contains an odd cycle, this is
a \emph{primitive} matrix, that is, there is $n_0$ such the $\ol P^n$
is strictly positive in each entry for all $n \ge 0$. Let $\bar \la$be 
the Perron-Frobenius eigenvalue of $\bar P$. Once more by the 
Perron-Frobenius theorem, all its other eigenvalues have absolute
value $< \bar \la$.

For all $z$, we can decompose $Q(z) = z \,\ol P + \ol Q(z)$,
where the matrix $\ol Q(z)$ is non-zero and non-negative for $z > 0$.
We have $|z| \,\ol P \le Q(|z|)$ elementwise, and the inequality is
strict in some entries. Therefore $|z| \bar \la < \la(|z|)$, and
$\la\,I - z\, \ol P$ is invertible, with
$$
\bigl(\la\,I - z\, \ol P\bigr)^{-1} 
= \sum_{n=0}^{\infty} \frac{z^n}{\la^{n+1}}\ol P^n\,.
$$
Since all entries of $\ol P^n$ are $>0$ for $n \ge n_0$, we have the strict
triangle inequality 
$$
\big|\bigl(\la\,I - z\, \ol P\bigr)^{-1}\big| <  
\sum_{n=0}^{\infty} \frac{|z|^n}{|\la|^{n+1}}\ol P^n = 
\bigl(\la(|z|)\,I - |z|\, \ol P\bigr)^{-1}\,.
$$
Here and below, the absolute value is meant to be taken matrix-element-wise.
Now we write
$$
\la\, I - Q(z) = \bigl(\la\,I - z\, \ol P\bigr)
\Bigl(I - \bigl(\la\,I - z\, \ol P\bigr)^{-1}\ol Q(z)\Bigr)\,.
$$
By the above strict inequality,
$$
\big|\bigl(\la\,I - z\, \ol P\bigr)^{-1}\ol Q(z)\big| <
\bigl(\la(|z|)\,I - |z|\, \ol P\bigr)^{-1} \ol Q(|z|)\,,
$$
a positive matrix whose Perron-Frobenius eigenvalue is $1$
(and the associated eigenvector is $\ww(|z|)$, the Perron-Frobenius 
eigenvector of $Q(|z|)$). We see that every eigenvalue of 
$\bigl(\la\,I - z\, \ol P\bigr)^{-1}\ol Q(z)$ has absolute value $< 1$.
Thus, $\la\, I - Q(z)$ as a product of two invertible matrices is
invertible, proving the claim.
 
\smallskip

We conclude that for $z \in [-R_{\mu}\,,\,0]$, every eigenvalue $\la$
of $Q(z)$ satisfies $|\la|<1$. Therefore $I - Q(z)$ is invertible, 
and the inverse matrix is $\sum_n Q(z)^n$.

\smallskip

In case (1), it follows that $\bigl(I - Q(z)\bigr)^{-1}$ is analytic 
at $z = -R_{\mu}\,$, whence \ref{eq:linsyst} yields that each 
function $G(x,y|z)$, where $x \in \partial_{\Ecal}\XX$ and $y=y_0\,$, 
is analytic at $-R_{\mu}\,$.
Arguing as at the end of the proof of Proposition \ref{pro:G-expand},
it follows that the same is true for arbitrary $x,y \in \XX$.

\smallskip

In case (2), we substitute $\zf = \sqrt{R+z}$ (for compex $z$, unless $z < -R$
is real). Then, again as in the the proof of Proposition \ref{pro:G-expand},
and using the same notation, we have for these $z$ that
$Q(z) = \Qss\bigl(\sqrt{R+z}\,\bigr)\,$, where the matrix 
$\Qss(\zf)$ is analytic and $I-\Qss(\zf)$ is invertible near $\zf=0$. 
We get
$$
\bigl(I- \Qss(\zf)\bigr)^{-1} = \Bss(\zf) 
= \sum_{n=0}^{\infty} \Qss(\zf)^n\,, 
$$
a matrix with analytic entries, and once more 
$\ff(z) = \Bss\bigl( \sqrt{R+z} \bigr) \ee$.
This yields the proposed singular expansion at $-R$. 

In order to show that
$|\bar h_{x,y}(-R_{\mu})| < h_{x,y}(R_{\mu})$ in case (c) of Proposition 
\ref{pro:G-expand}, it is sufficient to show that one has the element-wise
strict inequality
$|\Bss'(0)| < - \wt B'(0)$,
where $\wt B'(0)$ is as in \eqref{Btildepr}. Observe that the functions 
$f_V(z)$, $V \in \V_{\!\text{\rm ess}}$ are either even or odd, because 
each $C \in \Ecal$ is bipartite, that is, $f_V(-z) = \ep_V\,f_V(z)$
with $\ep_V \in \{ \pm 1\}$. It is then a straightforward task to compute
$$
\qss{\,}'_{T,U}(0) = - \sum_{T \vdash aVbU} \mu(a)\mu(b)\, R^2 \,\ep_V \, h_V(R)\,.
$$
By primitivity of $\wt Q(0) = Q(R)$, and since $\Qss(0) = Q(-R)$, 
we have once more the strict triangle inequality
$$
\left| \sum_{n=1}^{\infty} n\Qss(0)^{n-1} \right| < 
\sum_{n=1}^{\infty} n\wt Q(0)^{n-1}\,.
$$
The required inequality follows.

\smallskip

So far, this proof applies to $G(x,y|z)$ for all $x \in \partial_{\Ecal}\XX$ and
$y=y_0\,$.
Since the alternative between the three cases of Proposition  \ref{pro:G-expand}
has been shown to hold for \emph{all} $x,y \in X$, we get that also 
the alternative for the expansions at $-R_{\mu}$ is the same for all $x,y$.
\end{proof}

Now that we know the behaviour of $G(x,y|z)$ at all its singularities (i.e, 
at $R_{\mu}$ and possibly $-R_{\mu}$) on the circle of convergence,
the classical method of Darboux implies the following, compare with
\cite{La1} (and various other references).

\begin{thm}\label{thm:rw}
Suppose that the fully deterministic, symmetric labelled  graph $X$
admits a rooted, oriented tree set $\Ecal$ which
tesselates $X$ and is irreducible. Let $\mu$ be a probability measure
supported by $\Sigma$. Then one of the following alternatives holds
for the associated random walk. 
\begin{enumerate}
\item \hspace*{2cm} $p^{(n)}(x,y) \sim c_{x,y} \,R_{\mu}^{-n} \quad$ 
for all $\; x, y \in \XX$, or \\[4pt] 
\item 
\hspace*{2cm} $p^{(n)}(x,y) \sim c_{x,y} \,R_{\mu}^{-n}\, n^{-1/2} \quad$ 
for all $\; x, y \in \XX$, or  \\[4pt]
\item \hspace*{2cm} $p^{(n)}(x,y) 
\sim \bigl(c_{x,y} + (-1)^n \bar c_{x,y}\bigr)\,R_{\mu}^{-n}\, n^{-3/2} \quad$ 
for all $\; x, y \in \XX$,
\end{enumerate}
as $n \to \infty$ and $\dsf$ divides $n-d(x,y)$.  

Here, $\sim$ is asymptotic equivalence of sequences, that is,
quotients tend to $1$. We have $c_{x,y} > 0$ and $|\bar c_{x,y}| \le c_{x,y}$. 
Furthermore, $\bar c_{x,y}=0$ when $\dsf_s=\dsf$.
\end{thm}

We remark that one will be able to obtain analogous random walk 
asymptotics on pairs of groups whose
symmetric Schreier graphs have irreducible tree sets, also when the semigroup
presentation $\psi$ and the set $S = \psi(\Si)$ are \emph{not symmetric.}
There are some technical points that appear to make this
quite laborious and space-consuming when one follows the tree set approach,
compare with Remark \ref{rmk:non-symm}. On the other hand,  
here we have made an additional effort in order to clarify what happens
when the random walk is not necessarily strongly aperiodic (the latter means
that $\dsf_s=1$). In fact, those situations can arise quite naturally
when one considers general Schreier graphs.

Theorem \ref{thm:virtfree} yields the following.

\begin{cor}\label{cor:RW1}  
If $\psi$ is a symmetric semigroup
presentation of the virtually free group $G$, and $K$ is a finitely generated
subgroup of $G$, then for every random walk \eqref{eq:transprob}
on the Schreier graph $\XX(G,K,\psi)$, the transition probabilities 
satisfy (1), (2) or (3) of Theorem \ref{thm:rw}.
\end{cor}

\begin{proof} Let $\F$ be a free subgroup of $G$ with $[G:\F]<\infty$,
and let $\K = \F \cap K$. Then $k=[K:\K] < \infty$, and $\K$ is free.
By Theorem \ref{thm:virtfree}, the Schreier graph $X = X(G,\K,\psi)$
satisfies all hypotheses of Theorem \ref{thm:rw}. 
Let $\ol X = X(G,K,\psi)$. Write $o = \K$ and $\bar o = K$ for the
root vertices of $X$ and $\bar X$, respectively. For $x = \K g \in X$,
we have the natural projection $\si: X \to \ol X$, 
$x \mapsto \bar x = Kg \in \ol X$, which is $k$-to-one. Let $g_1, \dots, g_k$
be representatives of the right $\K$-cosets in $K$.  

Then we have 
$$ 
L_{\bar o, \bar x} = \{ w \in \Si^* : \psi(w) \in K \} 
= \{ w \in \Si^* : \psi(w) \in Kg_i \;\text{for some}\; i \le k  \} = 
\biguplus_{x \in X \,:\, \si(x) = \bar x} L_{o,x}\,.
$$ 
Therefore the transition probabilities $\bar p^{(n)}(\bar o,\bar x)$
of the random walk on $\bar X$ and $p^{(n)}(o,x)$ of the random walk on $X$ 
are related by
$$
\bar p^{(n)}(\bar o,\bar x) = \sum_{x \in X \,:\, \si(x) = \bar x} p^{(n)}(o,x)\,,
$$ 
a sum over $k$ elements. Thus, $p^{(n)}(\bar o,\bar x)$ inherits the
asymptotic behaviour of the $p^{(n)}(o,x)$, which has one of the
three above types. Since we can use any vertex as the root (which just means
passing to a conjugate of $K$), the result follows.
\end{proof}

It is an easy exercise to construct examples $(G,K)$ where $X(G,K)$
satisfies the assumptions of Theorem \ref{thm:rw}, but $K$ is not finitely
generated. 

By abuse of notation, we also consider $\mu$ as a probability measure on $G$.
In reality, in setting that we have chosen, this is the image under the 
mapping $\psi$ of the measure
$\mu$ on $\Si$. On $G$, the support is the symmetric set $A = \psi(\Si)$,
but $\mu$ itself is not necessarily symmetric. 

\begin{cor}\label{cor:RW2}
Let  $G$ be a finitely generated, virtually free group that is not virtually
cyclic,
and $\mu$ a probability measure on $G$ whose support is finite, symmetric, and
generates $G$. Then $R_{\mu} > 1$, and the transition probabilities of the associated
random walk satisfy
$$
p^{(n)}(x,y) \sim c_{x,y} \,R_{\mu}^{-n}\, n^{-3/2} \quad \text{for all}\quad
x, y \in G,
$$
as $n \to \infty$ (taking into account the parity of $n$, when the period is
$\dsf=2$.)
\end{cor}

\begin{proof} The fact that $R_{\mu} > 1$ follows from non-amenability of $G$
via a famous theorem of {\sc Kesten}~\cite{Ke}. 
See e.g. \cite[Cor. 12.5]{Wbook}.

It is known from a theorem of Guivarc'h \cite{Gu} that
the asymptotics of (1) and (2) in Theorem \ref{thm:rw} cannot occur on
a virtually free group that is not virtually cyclic. See e.g. 
\cite[Thm. 7.8]{Wbook}.
\end{proof}

As mentioned in the Introduction, this is closely related to the random 
walks on regular languages of \cite{La2}. 
Indeed, it is quite clear that a random walk on a virtually
free group can be interpreted in these terms, since those groups have
a normal form that is regular. (Underneath, there is a relation between
context-free graphs and finite state automata, compare also with 
\cite[\S 5]{CeWo2}.)
In order to apply the result of \cite{La2} to prove Theorem \ref{thm:rw},
one needs to provide a suitable regular normal form and then to work out that
the specific irreducibility hypotheses of
\cite{La2} are satisfied, a task that appears to have the same quintessence as 
proving existence of a tree set which is irreducible, as we have done.

\end{document}